\documentclass[reqno]{amsart}
\usepackage{fullpage}
\usepackage{amsfonts, amsmath, amssymb, amscd, amsthm}
\usepackage{mathrsfs}  
\usepackage{bbm}
\usepackage{color}
\usepackage{tikz} 
\usepackage{graphicx}  
\usepackage[all]{xy} \xyoption{arc} \xyoption{color}
\usepackage{epsfig}
\usepackage{amsbsy}
\usepackage{rotating}
\usepackage
[
a4paper,% other options: a3paper, a5paper, etc
vmargin=2cm,
hmargin=3cm
]{geometry}
\RequirePackage[bookmarks,%
colorlinks,%
urlcolor=prune,%
citecolor=red,%
linkcolor=blue,%
hyperfigures,%
pdfcreator=LaTeX,%
breaklinks=true,%
pdfpagelayout=SinglePage,%
bookmarksopen=true,%
bookmarksopenlevel=2]{hyperref}
\usepackage{float}
\usepackage{caption} 
\usepackage{subcaption}
\theoremstyle{plain}
\numberwithin{equation}{section}
\newtheorem{theorem}{Theorem}[section]

\theoremstyle{definition}

\def\ZZ{{\mathbb Z}}
\def\RR{{\mathbb R}}
\def\TT{{\mathbb T}}

\def\jump#1{{[\hspace{-2pt}[#1]\hspace{-2pt}]}}

\def\Bigjump#1{{\Big[\hspace{-4.5pt}\Big[#1\Big]\hspace{-4.5pt}\Big]}}

\def \x {\mb{x}}

\def \p{\partial}

\newcommand{\norm}[1]{\left\lVert#1\right\rVert}    
\newcommand\abs[1]{\left|#1\right|}    

\begin{document}
\title{Global existence and decay of the inhomogeneous Muskat problem with Lipschitz initial data}

\author{Diego Alonso-Or\'{a}n}
\address{Institute für Angewandte Mathematik, Universitat Bonn, Endenicher Allee 60, 53115 Bonn, Germany}
\email{alonso@iam.uni-bonn.de}
\author{Rafael Granero-Belinch\'on}
\address{Departamento  de  Matem\'aticas,  Estad\'istica  y  Computaci\'on,  Universidad  de Cantabria.  Avda.  Los  Castros  s/n,  Santander,  Spain.}
\email{rafael.granero@unican.es}

\subjclass[2010]{} 
\keywords{}

\date{\today}

\begin{abstract}
In this work we study the inhomogeneous Muskat problem, \emph{i.e.} the evolution of an internal wave between two different fluids in a porous medium with discontinuous permeability. In particular, under precise conditions on the initial datum and the physical quantities of the problem, our results ensure the decay of the solutions towards the equilibrium state in the Lipschitz norm. In addition, we establish the global existence and decay of Lipschitz solutions.
\end{abstract}
\maketitle
\tableofcontents
%%%%%%%%%%%%%%%%%%%%%%%%%
\section{Introduction and prior results}
Probably due to the ubiquitous nature of fluids, free boundary problems in fluid dynamics have attracted the attention of many different research groups in the last decades. These problems are physically interesting and mathematically challenging and, moreover, they are very connected to real world applications. 

In this paper we consider one of these problems, the so-called inhomogeneous Muskat problem (see \cite{granero2020growth} for a recent review). This problem studies the motion of two incompressible homogeneous fluids filling an inhomogeneous porous medium where the permeability is a step function. These fluids are assumed to be separated by a sharp interface (see Figure \ref{Figure1} below). Then, the plane is divided in the following three domains  
\begin{subequations} \label{domains}
\begin{align} 
\Omega^+(t)&=\{(x,y)\in \TT\times \RR,\;\; f(x,t)<y<\infty\} \,, \\
\Omega^-_{\text{top}}(t)&=\{(x,y)\in \TT\times \RR,\;\; -h_2<y<f(x,t)\}\,, \\
\Omega^-_{\text{bottom}}&=\{(x,y)\in \TT\times \RR,\;\; -\infty<y<-h_2\}\,,
\end{align} 
\end{subequations}
where $\TT$ denotes the unit circumference, namely, $\TT=[-\pi,\pi]$ with periodic boundary conditions. These domains are separated by the interfaces
\begin{subequations} \label{boundaries}
\begin{align} 
\Gamma(t)&=\{(x,f(x,t)),\;x_1\in\TT\} \,,\\
\Gamma_{\operatorname{perm}}&=\{(x,-h_2),\;x_1\in\TT\} \,,
\end{align} 
\end{subequations}
where the function $f$ and the constant $h_2$ satisfy
\begin{equation}\label{notouching}
\min_{x \in \TT  }f(x,t) >0  \ \text{ and } \ h_2>0.
\end{equation}
We observe that $f(x,t)$ is an unknown that must be determined from the dynamics while $h_2$ is part of the data of the problem. The fixed-in-time interface $\Gamma_{\operatorname{perm}}$ denotes the line accross which the permeability function $\kappa(x)$ has a jump, namely
$$
\kappa(x)=\left\{\begin{array}{ll}\kappa^+ & \text{ in }\Omega^+(t)\cup\Omega^-_{\text{top}}(t)\\
\kappa^- & \text{ in }\Omega^-_{\text{top}}(t)\cup\Omega^-_{\text{bottom}}\,\end{array}\right.,
$$
for given positive constants $\kappa^\pm>0$.  
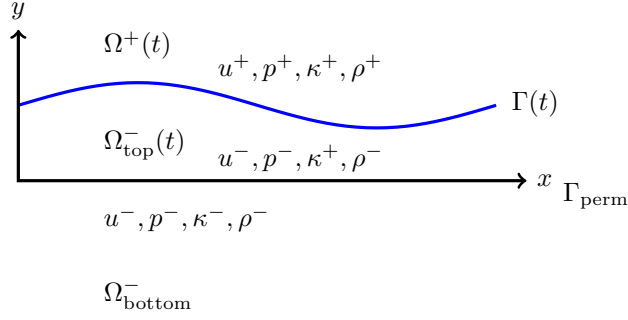
\begin{figure}[h]
	\centering
	\begin{tikzpicture}[domain=0:2*pi, scale=1] 
	\draw[color=black] plot (\x,{0.3*sin(\x r)+1}); 
	\draw[very thick, smooth, variable=\x, blue] plot (\x,{0.3*sin(\x r)+1}); 
	\draw[very thick,<->] (2*pi+0.4,0)  node[right] {$x$} -- (0,0) -- (0,2) node[above] {$y$};
	\coordinate[label=above:{$\Gamma_{\text{perm}}$}] (A) at (7.6,-0.5);
	\coordinate[label=above:{$\Gamma(t)$}] (A) at (6.8,0.72);
	\node[right] at (1,-0.5) {$u^-,p^-,\kappa^-,\rho^-$};
	\node[right] at (1,1.8) {$\Omega^+(t)$};
	\node[right] at (1,0.5) {$\Omega^-_{\text{top}}(t)$};
	\node[right] at (1,-1.5) {$\Omega^-_{\text{bottom}}$};
	\node[right] at (2.5,1.5) {$u^+,p^+,\kappa^+,\rho^+$};
	\node[right] at (2.5,0.3) {$u^-,p^-,\kappa^+,\rho^-$};
	\end{tikzpicture} 
	\caption{Physical setting}
	\label{Figure1}
\end{figure}
With this notation, assuming that the acceleration due to gravity is equal to 1 and denoting by $u,p$ and $\rho$ the velocity, pressure and density of the fluids, the equations describing the phenomenon are
\begin{subequations}\label{muskat}
\begin{alignat}{2}
\frac{u^+}{\kappa^+}+\nabla p^+&=-\rho^+ e_2,  \qquad&&\text{in } \Omega^+(t)\times[0,T]\,,\\
\frac{u^-}{\kappa^+}+\nabla p^-&=-\rho^- e_2,  \qquad&&\text{in } \Omega^-_{\text{top}}(t)\times[0,T]\,,\\
\frac{u^-}{\kappa^-}+\nabla p^-&=-\rho^- e_2,  \qquad&&\text{in } \Omega^-_{\text{bottom}}\times[0,T]\,,\\
\nabla\cdot u^+ &=0,  \qquad&&\text{in } \Omega^+(t)\times[0,T]\,,\\
\nabla\cdot u^- &=0,  \qquad&&\text{in } \Omega^-_{\text{top}}(t)\times[0,T]\,,\\
\nabla\cdot u^- &=0,  \qquad&&\text{in } \Omega^-_{\text{bottom}}\times[0,T]\,,\\
\jump{p} &= 0 &&\text{on }\Gamma(t)\times[0,T],\\
\jump{p} &= 0 &&\text{on }\Gamma_{\operatorname{perm}}\times[0,T],\\
\jump{\nabla p\cdot e_2}  &= -\Bigjump{\frac{1}{\kappa}}u^-\cdot e_2 \qquad &&\text{on }\Gamma_{\operatorname{perm}}\times[0,T],
\end{alignat}
\end{subequations}
where $n$ is the upward pointing unit normal to $\Gamma(t)$ and 
 $\jump{w} = w^+ - w^-$ denotes the jump of a discontinuous function $w$ across $\Gamma(t)$ or $\Gamma_{\operatorname{perm}}$. 

As it is well-known, this system can be written equivalently using the contour equation formulation \cite{berselli2014local}. In this case, the inhomogeneous Muskat problem \eqref{muskat} is equivalent to the following nonlinear and nonlocal PDE for the free boundary $f(x,t)$: 
\begin{multline}\label{eq:I}
\p_{t} f(x)=\frac{\kappa^+(\rho^--\rho^+)}{4\pi} \text{P.V.}\int_\TT\frac{\sin(\beta)(\p_{x} f(x)-\p_{x} f(x-\beta))}{\cosh(f(x)-f(x-\beta))-\cos(\beta)} d\beta\\
+\frac{1}{4\pi} \text{P.V.}\int_\TT\frac{(\p_{x} f(x)\sinh(f(x)+h_2)+\sin(x-\beta))\varpi_2(\beta)}{\cosh(f(x)+h_2)-\cos(x-\beta)}  d\beta,
\end{multline}
where the second vorticity amplitude $\varpi_2(x)$ can be written as
\begin{equation}
\varpi_2(\beta)=\mathcal{A}\text{P.V.}\int_\TT\frac{\sinh(h_2+f(\gamma))\p_{x}f(\gamma)}{\cosh(h_2+f(\gamma))-\cos(\beta-\gamma)}  d\gamma,\label{eq:II}
\end{equation}
and we have defined the parameter 
$$
\mathcal{A}:=\frac{\kappa^+(\rho^--\rho^+)}{2\pi}\frac{\kappa^+-\kappa^-}{\kappa^++\kappa^-}.
$$
This contour equation can be written in divergence form as
\begin{multline}\label{eq:divergence}
\p_{t} f(x)=\frac{\kappa^+(\rho^--\rho^+)}{\pi} \p_{x}\text{P.V.}\int_\TT \arctan\left( \frac{\tanh((f(x)-f(x-\beta))/2)}{\tan(\beta/2)}\right) \ d\beta\\
+\frac{1}{4\pi} \p_{x}\text{P.V.}\int_\TT\log(\cosh(f(x)+h_2)-\cos(x-\beta))\varpi_2(\beta)d\beta,
\end{multline}
where, furthermore, $\varpi_2$ is a zero order operator in $f$. Indeed, we can compute
\begin{align*}
\varpi_2(\beta)&=\mathcal{A} \text{P.V.}\int_\TT\frac{\sinh(h_2+f(\gamma))\p_{x}f(\gamma)d\gamma}{\cosh(h_2+f(\gamma))-\cos(\beta-\gamma)}\pm\frac{\sin (\beta-\gamma)}{\cosh(h_{2}+f(\gamma))-\cos(\beta-\gamma)}\\
&=\mathcal{A} \text{P.V.}\int_\TT\partial_\gamma \log(\cosh(h_2+f(\gamma))-\cos(\beta-\gamma))+\mathcal{A} \text{P.V.}\int_\TT\frac{\sin (\beta-\gamma)}{\cosh(h_{2}+f(\gamma))-\cos(\beta-\gamma)}\\
&=\mathcal{A} \text{P.V.}\int_\TT\frac{\sin (\beta-\gamma)}{\cosh(h_{2}+f(\gamma))-\cos(\beta-\gamma)}.
\end{align*}
As a consequence, it admits a straightforward notion of weak solution. Namely, $f$ is a weak solution of \eqref{eq:divergence} (and consequently of \eqref{muskat}) if and only if for all $\phi(x,t)\in C_c^\infty([0,T)\times \mathbb{T})$ the following equality holds
\begin{multline}\label{eq:weak}
-\int_\mathbb{T}\phi(x) f_0(x)dx-\int_0^T\int_\mathbb{T}\p_{t}\phi f(x)dxds\\
+\frac{\kappa^+(\rho^--\rho^+)}{\pi} \int_0^T\int_\mathbb{T}\p_{x}\phi(x)\text{P.V.}\int_\TT  \arctan\left( \frac{\tanh((f(x)-f(x-\beta))/2)}{\tan(\beta/2)}\right) \ d\beta dx ds\\
+\frac{1}{4\pi} \int_0^T\int_\mathbb{T}\partial_x \phi(x) \text{P.V.}\int_\TT  \log(\cosh(f(x)+h_2)-\cos(x-\beta))\varpi_2(\beta)d\beta dx ds=0.
\end{multline}

The Muskat problem is a popular research topic and the literature is huge. For an overview of the different results available we refer to \cite{ambrose2004well,ambrose2014zero,cheng2016well,cordoba2011interface,
cordoba2013porous,cordoba2007contour,escher2011generalized,escher2011parabolicity,
escher2015domain,granero2020growth,matioc2019well,matioc2018viscous,pruess2016muskat,siegel2004global} and the references therein.

The Muskat problem is scale invariant under the transformation
$$
f_\lambda(x,t)=\frac{1}{\lambda}f(\lambda x,\lambda t).
$$
This scaling serves as a zoom in towards the small scales and implies that the $L^2-$based Sobolev spaces $H^r$ for $r>3/2$ are subcritical. In this range of spaces a local solution is known to exists for arbitrary initial data (see for instance \cite{alazard2020paralinearization,ambrose2004well,cordoba2011interface,cordoba2013porous,cordoba2007contour,matioc2019well,nguyen2020paradifferential,pernas2017local}). 

As a byproduct of the previous scaling, we can define a number of critical spaces. Some of them are the Sobolev space $H^{3/2}$, the Lipschitz class $W^{1,\infty}$, the space of $C^1$ functions or the Wiener space 
$$
A^1=\{f\text{ s.t. }i\xi \hat{f}(\xi)\in L^1\},
$$
among others. To develop new ideas and tools to handle the very difficult case of initial data in critical spaces is a very hot research area in PDEs. In this regard, we refer to the works by C\'ordoba \& Gancedo \cite{cordoba2007contour} and Constantin, C\'ordoba, Gancedo \& Strain\cite{constantin2012global}, Constantin, C\'ordoba, Gancedo, Rodr\'iguez-Piazza \& Strain \cite{constantin2016muskat}, Gancedo, Garc\'ia-Ju\'arez, Patel \& Strain\cite{gancedo2019muskat}, Patel \& Strain \cite{patel2017large} and Gancedo, Granero \& Scrobogna \cite{gancedo2020surface} for the proof of the global existence of strong solution for small initial data in the Wiener space $A^1$.

Similarly, we refer to the works by C\'ordoba \& Lazar \cite{cordoba2018global} (see also \cite{granero2020growth}), Gancedo \& Lazar \cite{gancedo2020global} and Alazard \& Nguyen \cite{alazard2020endpoint,alazard2021cauchy,alazard2021cauchy2,alazard2021quasilinearization} for results regarding the space $H^{3/2}$. In these papers the global existence for small initial data in the critical space $H^{3/2}$ were proved furthermore, the local existence for arbitrary initial data in the critical space has been also established. Let us emphasize that in some of the previous results the initial data is also assumed to be Lipschitz.

Remarkably, certain solutions to the Muskat problem experience turning singularities \cite{castro2013breakdown,castro2012rayleigh,cordoba2015note,cordoba2017note,gomez2014turning}, \emph{i.e.} a finite time blow up of the derivative of the interface $f$, and splash singularities \cite{castro2016splash}, \emph{i.e.} a self intersection of the interface. Then, the Lipschitz norm seems of particular importance for the Muskat problem. In this regards, we refer to the pioneering work of C\'ordoba \& Gancedo \cite{cordoba2009maximum} (see also \cite{constantin2016muskat,constantin2012global,granero2014global}) where a maximum principle for the $W^{1,\infty}$ was established. Similarly, we refer to the work Constantin, Gancedo, Shvydkoy \& Vicol \cite{constantin2017global} for a proof of the global existence of solution corresponding to small initial data in the Lipschitz class. Similar results have been obtained by Cameron \cite{cameron2018global,cameron2020eventual,cameron2020global}. In the case of $C^1$ initial data the interested reader is referred to the work by Chen, Nguyen, \& Xu  \cite{chen2021muskat}.
 
The cases of porous media with discontinuous permeability or impervious boundaries have considered in a number of previous works. In particular, the known results covering the local existence and finite time singularities are contained in the papers \cite{berselli2014local,cordoba2014confined,gomez2014turning,pernas2017local}. However, to the best of our knowledge, the only global results were obtained in \cite{granero2014global,granero2019well,patel2021global}. Remarkably, in the case of a discontinuous permeability the only result in critical spaces is the paper by Patel \& Shankar \cite{patel2021global} where the authors consider initial data in Wiener space $A^1$. 

Numerical simulations for the inhomogeneous Muskat problem were carried over in the paper \cite{berselli2014local}. Those simulations suggest that the amplitude of the internal wave decays towards the equilibrium at least for most of the initial data. However, a rigorous proof or a counterexample of the possible decay towards equilibrium in Lipschitz norm remain as an open question in the case of a porous medium with two different permeabilities. Furthermore, the global in time existence for small Lipschitz initial data was also another open problem so far. The purpose of this paper is to clarify the previously mentioned questions. Specifically, in this work we establish conditions on the initial data and the physical parameters of the problem that ensure the decay of the solution in $W^{1,\infty}$ towards the equilibrium state. In addition, we also prove the global existence of weak solutions and their decay towards the flat equilibrium.

Besides the fact that the problem is \emph{more nonlocal} due to the permeability jump, the inhomogeneous Muskat problem is more challenging also due to the fact that there is a new scenario of possible pathological behaviour. Indeed, when the internal wave separating both fluids reaches the curve where the permeability changes its value, the contour equation becomes more defiant in the sense that the term with the double integral becomes singular. In particular, our results rule out the previously mentioned scenario (see the figure below) where the internal wave reaches the curve where the permeability changes and, at the same time, discard the possibility of turning and splash singularities. 

\begin{figure}[h]
	\centering
	\begin{tikzpicture}[domain=0:2*pi, scale=1] 
    \draw[very thick, smooth, variable=\x, red] plot (\x,{0.3*sin(\x r)+0.3}); 
	\draw[very thick, smooth, variable=\x, blue] plot (\x,{0.3*cos(\x r)+1}); 
	\draw[very thick,<->] (2*pi+0.4,0) -- (0,0) -- (0,2);
	\coordinate[label=above:{$\Gamma_{\text{perm}}$}] (A) at (7.6,-0.5);
	\coordinate[label=above:{$\Gamma(t_0)$}] (A) at (6.8,0.85);
	\coordinate[label=above:{$\Gamma(t_1)$}] (A) at (6.8,0.17);
	\end{tikzpicture} 
	\caption{Internal wave touching the curve where the permeability changes.}
\end{figure}
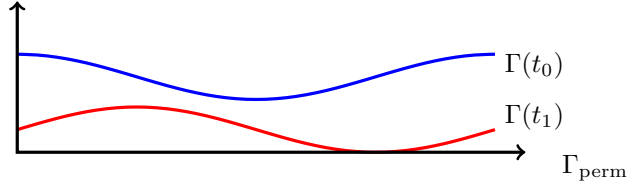

\subsection{Main results}
The purpose of this paper is to study the global existence of weak solutions and the trend to equilibrium of the inhomogeneous Muskat problem in the Lipschitz class. In particular, our first result shows the decay of the $L^\infty$ norm and reads as follows:
\begin{theorem}\label{teo1}
Let $0< f\in C([0,T],H^{3}(\TT))$ be a smooth solution to \eqref{eq:I}, assume that the system is in the Rayleigh-Taylor stable case $\rho^--\rho^+>0$ and that $h_2>0$ and define
$$
\langle f(x,0)\rangle=\frac{1}{2\pi}\int_\TT f(x,0)dx.
$$
Then if 
\begin{multline}\label{condition:smallness}
 \frac{1}{\cosh((\displaystyle\max_x f(x,0)-(\displaystyle\min_x f(x,0))/2)}\\
    \\-\frac{8\mathcal{|\frac{\kappa^+-\kappa^-}{\kappa^++\kappa^-}|}}{\pi^2 (\cosh(h_2+\langle f(x,0)\rangle)-1)^{3}} \ \sinh(\langle f(x,0)\rangle+\max_x f(x,0)-\min_x f(x,0)+h_2)>0,
\end{multline}
we have that
\begin{equation}\label{osc:decay}
(\max_x f(x,t)-\min_x f(x,t))\leq (\max_x f(x,0)-\min_x f(x,0))e^{-\mu t}
\end{equation}
for a sufficiently small constant $0<\mu$ depending only on the physical parameters of the problem and the initial data. In particular, the solution remains non-negative and we have the following bound
$$
\|f(t)\|_{L^\infty}\leq \langle f(x,0)\rangle+(\max_x f(x,0)-\min_x f(x,0))e^{-\mu t}.
$$
\end{theorem}

Our second result deals with the $\dot{W}^{1,\infty}$ maximum principle:
\begin{theorem}\label{teo2}
Let $0< f\in C([0,T],H^{3}(\TT))$ be a smooth solution to \eqref{eq:I}, assume that the system is in the Rayleigh-Taylor stable case $\rho^--\rho^+>0$ and that $h_2>0$. Let us furthermore assume that the initial data satisfies the hypothesis of Theorem \ref{teo1} and, in addition, that
\begin{align}
0>&-  \frac{\kappa^+(\rho^--\rho^+)}{4}\|\partial_x f(0)\|_{L^\infty} \left(1-\mathcal{P}(\|\partial_x f(0)\|_{L^\infty})\right)+ \pi\abs{\mathcal{A}} \frac{\cosh(\norm{f(0)}_{L^{\infty}}+h_{2})}{(\cosh(h_{2})-1)^2}\|\partial_x f(0)\|_{L^\infty}^2 \nonumber\\
&\quad + 2\pi \abs{\mathcal{A}}\frac{\sinh^{2}(\norm{f(0)}_{L^{\infty}}+h_2)}{(\cosh(h_{2})-1)^{3}}\|\partial_x f(0)\|_{L^\infty}^2+2\pi\abs{\mathcal{A}}\frac{\sinh(\norm{f(0)}_{L^{\infty}}+h_{2})}{(\cosh(h_{2})-1)^3}\|\partial_x f(0)\|_{L^\infty}, \label{eq:hipnorminf}
\end{align}
with
$$
\mathcal{P}(z)=\left[\frac{1}{\left(1-\left(\sinh(z(\pi/2))\right)^{2}\right)^2}-1\right]+\frac{1}{2}z(\cosh(\pi z) 2\pi+\pi \sinh(\pi z)).
$$
Then, the solution satifies
$$
\|\partial_x f(t)\|_{L^\infty}\leq \|\partial_x f(0)\|e^{-\mu t},
$$
for a sufficiently small constant $\mu$ depending on the parameters of the problem and the initial data.
\end{theorem}

Finally, we state our result ensuring the global existence and decay of solutions to the inhomogeneous Muskat problem \eqref{eq:I}.
\begin{theorem}\label{teo3}
Let $0< f_0\in W^{1,\infty}$ be the initial data for the inhomogeneous Muskat problem \eqref{eq:I} satisfying the hypotheses of Theorems \ref{teo1} and \ref{teo2}. Let us consider that the system is in the Rayleigh-Taylor stable case $\rho^--\rho^+>0$ and that $h_2>0$. There exists a constant $c$ such that if
$$
\|\partial_x f_0\|_{L^\infty}\leq c,
$$
there exist at least one global weak solution (in the sense of definition \eqref{eq:weak})
$$
f\in L^\infty(0,T;W^{1,\infty})\cap L^2(0,T;H^{3/2}) \quad \forall\, T>0.
$$
Furthermore, this solution satisfies 
$$
\|f(t)\|_{L^\infty}\leq \langle f(x,0)\rangle+(\max_x f(x,0)-\min_x f(x,0))e^{-\mu t}.
$$
$$
\|\partial_x f(t)\|_{L^\infty}\leq \|\partial_x f(0)\|e^{-\mu t},
$$
for a sufficiently small constant $\mu$ depending on the parameters of the problem and the initial data.
\end{theorem}

\subsection{Notation}
We will use the following notation throughout the manuscript. We are working on the one dimensional torus $\TT=[-\pi,\pi]$, endowed with periodic boundary conditions. For  $1\leq p<\infty$ we denote by  $L^p(\TT;\RR)$ the standard Lebesgue space of measurable $p$-integrable $\RR$-valued functions and by  $L^\infty(\TT;\RR)$ the space of essentially bounded functions. Particularly, $L^2(\TT;\RR)$ is equipped with the inner product 
	$
	(f,g)_{L^2}=\int_{\TT}f\cdot\overline{g}\,{\rm d}x,
	$
	where $\overline{g}$ denotes the complex conjugate of $g$.  The Fourier
	coefficients and the Fourier series of $f(x)\in L^2(\TT;\RR)$ are defined by
	$\widehat{f}(\xi)=\int_{\TT}f(x){\rm e}^{-{\rm i}x \xi}\,{\rm d}x$ and 
	$f(x)=\frac{1}{2\pi}\sum_{k\in{\mathbb Z}}\widehat{f}(k){\rm e}^{{\rm i}x k}$, respectively. We define the $L^2$-based Sobolev space $H^s$ on $\TT$ with values in $\RR$ as
	\begin{align*}
		H^s(\TT;\RR):=\left\{f\in L^2(\TT;\RR):\|f\|_{H^s(\TT;\RR)}^2
		=\sum_{k\in{\mathbb Z}}|(1+|k|^2)^{s/2} f(k)|^2<+\infty\right\}.
	\end{align*}
In case $s\in\ZZ^{+}$, the Sobolev space $W^{s,p}(\TT;\RR)$ is defined as
$$ W^{s,p} (\TT)=\{f\in L^{p}(\TT), \p_{x}^{s}f\in L^{p}(\TT)\} $$
endowed with the norm 
$$\norm{f}^{p}_{W^{s,p}(\TT)}= \norm{f}^{p}_{L^p}+\norm{\p_{x}^{s}f}^{p}_{L^p}$$
where $\partial_{x}$ is the classical first order differential operator with respect $x$. Their homogenous counterparts will be denoted by  $\dot{W}^{s,p} (\TT,\RR)$ and $\dot{H}^{s}(\TT,\RR)$ when $p=2$. To simplify notation, we will just write
$$ L^p=L^p(\TT,\RR), \quad H^s=H^s(\TT,\RR), \quad W^{s,p}=W^{s,p} (\TT,\RR),\quad \dot{H}^s=\dot{H}^s(\TT,\RR), \quad \dot{W}^{s,p}=\dot{W}^{s,p} (\TT,\RR).$$
We denote with $C=C(\cdot)$ any positive universal constant that may depend on fixed parameters and controlled quantities. Note also that this constant may vary from line to line. It is also important to remind that in order to light notation the condition almost everywhere (a.e) is not always indicated.

\subsection{Plan of the paper}
We conclude the introduction by outlining the contents of this paper. In Section 2 we present the proof of Theorem \ref{teo1} showing the decay in the $L^\infty$ norm of the solution. In Section 3 we prove the decay in the $\dot{W}^{1,\infty}$ norm of the solution providing the proof of Theorem \ref{teo2}. Later, in Section 4 we prove the existence of global weak solutions. To that purpose, we first regularize the problem and show a priori estimates of the approximated solutions in $H^3$. Furthermore, the approximated solutions also satisfy the pointwise estimates obtained in Theorems \ref{teo1} and \ref{teo2}. Next, we show uniform $H^1$ estimates which
are of lower order but allow us to use the parabolic gain of regularity ensuring the needed compacity to pass to the limit and establishing Theorem \ref{teo3}.

\section{Decay in $L^\infty$}\label{decay:Linf}
In this section we establish the decay in the $L^\infty$ norm of the solution.

\begin{proof}[Proof of Theorem \ref{teo1}]
The proof follows from a pointwise method. This pointwise approach allow us to track the dynamics for the oscillation of the function $f$ and will ensure that $f$ decays towards the equilibrium and maintains its positivity.

Without losing generality we assume that
$$
\int_{\TT}f(x) \ dx =2\pi.
$$
Then, we can write
$$
f=1+g, 
$$
where $g$ has zero mean, i.e. 
$$
\int_{\TT}g(x) \ dx =0.
$$

Before showing the oscillation decay \eqref{osc:decay}, let us rewrite system \eqref{eq:I}-\eqref{eq:II}
in a much more convenient way. To that purpose, we recall the divergence form of the Muskat problem \eqref{eq:divergence}. Thus, expanding the derivative we have that 
\begin{align}
\p_{t}f(x)=&{}\frac{\kappa^+(\rho^--\rho^+)}{2\pi} \p_{x}f(x)\text{P.V.}\int_{\TT} \frac{\tan((x-\beta)/2)(1-\tanh^2((f(x)-f(\beta)/2))}{\tan^{2}((x-\beta)/2)+\tanh^{2}((f(x)-f(\beta)/2))} \ d\beta  \nonumber \\
     &{}\hspace{0.5cm}-\frac{\kappa^+(\rho^--\rho^+)}{2\pi} \text{P.V.}\int_{\TT} \frac{\tanh((f(x)-f(\beta))/2)(1+\tan^{2}((x-\beta)/2)}{\tan^{2}((x-\beta)/2)+\tanh^{2}((f(x)-f(\beta)/2))} \ d\beta\nonumber\\
     &+\frac{\mathcal{A}}{4\pi} \text{P.V.}\int_\TT\int_{\TT}\frac{(\p_{x} f(x)\sinh(f(x)+h_2)+\sin(w))}{\cosh(f(x)+h_2)-\cos(w)}
 \frac{\sin(z)}{\cosh(h_{2}+f(x-w-z))-\cos(z)}  dz \ dw \label{eq:simplified:1}
\end{align}
where we have used the change of variables $x-\beta=w$ and $\beta-\gamma=z$. Using the fact that
$$ \text{P.V.}\int_\TT \frac{\sin(z)}{\cosh(f(x-w)+h_{2})-\cos(z)} dz=0$$ 
we rewrite the last integral in \eqref{eq:simplified:1} as
\begin{align}
\frac{\mathcal{A}}{4\pi} \text{P.V.}&\int_\TT\int_{\TT}\frac{(\p_{x} f(x)\sinh(f(x)+h_2)+\sin(w))}{\cosh(f(x)+h_2)-\cos(w)}
 \frac{\sin(z)}{\cosh(h_{2}+f(x-w-z))-\cos(z)}  dz \ dw \nonumber\\
={}&\frac{\mathcal{A}}{4\pi} \text{P.V.}\int_\TT\int_{\TT}\left[ \frac{(\p_{x} f(x)\sinh(f(x)+h_2)+\sin(w))}{\cosh(f(x)+h_2)-\cos(w)}
\right]  \nonumber \\ 
&\hspace{1cm} \times \left[\frac{\sin(z)}{\cosh(h_{2}+f(x-w-z))-\cos(z)}-\frac{\sin(z)}{\cosh(f(x-w)+h_{2})-\cos(z)}\right] dz \ dw \nonumber \\ 
=&{}-\frac{\mathcal{A}}{4\pi} \text{P.V.}\int_\TT\int_{\TT}\left[ \frac{(\p_{x} f(x)\sinh(f(x)+h_2)+\sin(w))}{\cosh(f(x)+h_2)-\cos(w)}\right]   \nonumber \\
& \hspace{1cm} \times \frac{\sin(z)\left( \cosh(f(x-w-z)+h_{2})-\cosh(f(x-w)+h_{2})\right)}{\left[ \cosh(h_{2}+f(x-w-z))-\cos(z)\right]\left[\cosh(h_{2}+f(x-w))-\cos(z)\right]}  dz \ dw    \nonumber \\
={}&\frac{\mathcal{A}}{4\pi} \text{P.V.}\int_\TT\int_{\TT}\left[ \frac{(\p_{x} f(x)\sinh(f(x)+h_2)+\sin(w))}{\cosh(f(x)+h_2)-\cos(w)}
\right]  \nonumber \\
 &\hspace{0.2cm} \times \frac{\sin(z)(f(x-w-z)-f(x-w))}{\left[ \cosh(h_{2}+f(x-w-z))-\cos(z)\right]\left[\cosh(h_{2}+f(x-w))-\cos(z)\right]}\nonumber\\
  &\hspace{0.2cm} \times\int_{0}^{1}\sinh\left[\lambda (f(x-w-z)+h_{2})-(1-\lambda)(f(x-w)+h_2) \right]\ d\lambda \ dz \ dw. \label{eq:simplified:2}
\end{align}

Due to the smoothness of $f$, we have that $g\in C^1([0,T]\times\TT)$. In particular,
\begin{equation}\label{max:g}
M(t)=\displaystyle\max_{x}g(x,t)=g(\overline{x}_t,t)>0
\end{equation}
and 
\begin{equation}\label{min:g}
m(t)=\displaystyle\min_{x}g(x,t)=g(\underline{x}_t,t)<0
\end{equation}
are Lipschitz functions. Indeed, to see this it is enough to observe that $M(t)$ satisfies

\begin{align*}
\left|M(t)-M(s)\right|&=\left\{\begin{array}{cc}g(\overline{x}_t,t)-g(\overline{x}_s,s) \text{ if } M(t)>M(s)\\
g(\overline{x}_s,s)-g(\overline{x}_t,t) \text{ if }M(s)>M(t)\end{array}\right.
\\
&\leq \left\{\begin{array}{cc}g(\overline{x}_t,t)-g(\overline{x}_t,s) \text{ if } M(t)>M(s)\\
g(\overline{x}_s,s)-g(\overline{x}_s,t) \text{ if }M(s)>M(t)\end{array}\right.
\\
&\leq \left\{\begin{array}{cc}|\partial_t g(\overline{x}_t,z)||t-s| \text{ if } M(t)>M(s)\\
|\partial_t g(\overline{x}_s,z)||t-s| \text{ if }M(s)>M(t)\end{array}\right.
\\
&\leq \max_{y,z}|\partial_t g(y,z)||t-s|.
\end{align*}
For $m(t)$ we can compute similarly and conclude
$$
\left|m(t)-m(s)\right|\leq \max_{y,z}|\partial_t g(y,z)||t-s|.
$$
We invoke Rademacher's Theorem to obtain that $M(t)$ and $m(t)$ are differentiable (in $t$) almost everywhere. Then we have that
\begin{align*}
M'(t)&=\lim_{h\rightarrow0} \frac{g(\overline{x}_{t+h},t+h)-g(\overline{x}_t,t)}{h}\\
&=\lim_{h\rightarrow0} \frac{g(\overline{x}_{t+h},t+h)-g(\overline{x}_t,t)\pm g(\overline{x}_{t+h},t)}{h}\\
&\leq\lim_{h\rightarrow0} \frac{g(\overline{x}_{t+h},t+h)-g(\overline{x}_{t+h},t)}{h}\\
&\leq \partial_t g(\overline{x}_{t},t).
\end{align*}
Similarly,
\begin{align*}
M'(t)&=\lim_{h\rightarrow0} \frac{g(\overline{x}_{t+h},t+h)-g(\overline{x}_t,t)}{h}\\
&=\lim_{h\rightarrow0} \frac{g(\overline{x}_{t+h},t+h)-g(\overline{x}_t,t)\pm g(\overline{x}_{t},t+h)}{h}\\
&\geq\lim_{h\rightarrow0} \frac{g(\overline{x}_{t},t+h)-g(\overline{x}_{t},t)}{h}\\
&\geq \partial_t g(\overline{x}_{t},t).
\end{align*}
As a consequence
$$
M'(t)=\partial_t g(\overline{x}_{t},t) \text{ a.e.}
$$
Similarly, we can obtain that
$$
m'(t)=\partial_t g(\underline{x}_{t},t) \text{ a.e.}
$$

Therefore, taking $x=\overline{x}_t$ in equation \eqref{eq:simplified:1} and using that $\p_{x}g(\overline{x}_{t},t)=0$, we infer that 
\begin{align*}
M'(t)=&{}-\frac{\kappa^+(\rho^--\rho^+)}{2\pi} \text{P.V.}\int_{\TT} \frac{\tanh((g(\overline{x}_t)-g(\beta))/2)(1+\tan^{2}((\overline{x}_t-\beta)/2)}{\tan^{2}((\overline{x}_t-\beta)/2)+\tanh^{2}((g(\overline{x}_t)-g(\beta)/2))} \ d\beta\nonumber\\
     &+\frac{\mathcal{A}}{4\pi} \text{P.V.}\int_\TT\int_{\TT}\frac{\sin(w)}{\cosh(1+g(\overline{x}_t)+h_2)-\cos(w)}
 \frac{\sin(z)}{\cosh(h_{2}+1+g(\overline{x}_t-w-z))-\cos(z)}  dz \ dw
\end{align*}
Using the previous calculations in \eqref{eq:simplified:2}, we can further write
\begin{align*}
&M'(t)={} -\frac{\kappa^+(\rho^--\rho^+)}{2\pi} \text{P.V.}\int_{\TT} \frac{\tanh((M(t)-g(\beta))/2)(1+\tan^{2}((\overline{x}_t-\beta)/2)}{\tan^{2}((\overline{x}_t-\beta)/2)+\tanh^{2}((M(t)-g(\beta)/2))} \ d\beta \nonumber \\
&\hspace{1cm} + \frac{\mathcal{A}}{4\pi} \text{P.V.}\int_\TT\int_{\TT}\left[ \frac{\sin(w)}{\cosh(M(t)+1+h_2)-\cos(w)}
\right]  \nonumber \\
 & \times \frac{\sin(z)(g(\overline{x}_t-w-z)-g(\overline{x}_t-w))}{\left[ \cosh(h_{2}+g(\overline{x}_t-w-z)+1)-\cos(z)\right]\left[\cosh(h_{2}+g(\overline{x}_t-w)+1)-\cos(z)\right]}  dz \ dw \nonumber\\
  & \times \int_{0}^{1}\sinh\left[\lambda (g(\overline{x}_t-w-z)+1+h_{2})-(1-\lambda)(g(\overline{x}_t-w)+1+h_2) \right]\ d\lambda  \nonumber \\
 &:= K_{1}+K_{2}.
\end{align*}
Taking into account that
 \begin{align*}
 \tan^{2}((\overline{x}_t-\beta)/2)+\tanh^{2}((M(t)-g(\beta)/2)) &\leq \tan^{2}((\overline{x}_t-\beta)/2+1, \\
 \sinh((M(t)-g(\beta))/2) & \geq (M(t)-g(\beta))/2,
 \end{align*}
 we have that $K_1$ is bounded by
\begin{align*}
     K_1 &\leq -\frac{\kappa^+(\rho^--\rho^+)}{2\pi} \int_{\TT} \frac{(M(t)-g(\beta))/2)}{\cosh((M(t)-g(\beta)/2))} \ d\beta \\
     &\leq -\frac{\kappa^+(\rho^--\rho^+)}{2\pi} \frac{1}{{\cosh((M(t)-m(t)/2))}} \int_{\TT} (M(t)-g(\beta))/2) \ d\beta  \\
    &\leq -\frac{\kappa^+(\rho^--\rho^+)}{2}\frac{M(t)}{\cosh((M(t)-m(t)/2))}
 \end{align*}
due to the mean zero condition. To estimate the term $K_2$ it suffices to notice that 
$$
     \cosh(M(t)+1+h_{2})-\cos(z) \geq \cosh(1+h_2)-1
$$
and hence 
\begin{align*}
     K_2 &\leq \frac{|\mathcal{A}|}{2\pi} \frac{(M(t)-m(t))\sinh(1+M(t)+h_2)}{(\cosh(h_2+1)-1)^{3}} \int_\TT\int_{\TT} \abs{\sin(w)\sin(z)}\ dz dw  \\
     &\leq \frac{8|\mathcal{A}|}{\pi (\cosh(h_2+1)-1)^{3}} \left(M(t)-m(t)\right)\sinh(1+M(t)-m(t)+h_2),
 \end{align*}
where we have used that 
$$
-m(t)>0.
$$
Thus, collecting both estimates we have shown that
\begin{multline}\label{ODE:M(t)}
M'(t) \leq -\frac{\kappa^+(\rho^--\rho^+)}{2}\frac{M(t)}{\cosh((M(t)-m(t)/2))}\\
+\frac{8|\mathcal{A}|}{\pi (\cosh(h_2+1)-1)^{3}} \left(M(t)-m(t)\right)\sinh(1+M(t)-m(t)+h_2))\text{ a.e.}
\end{multline}
Similarly, we can compute
\begin{align*}
&-m'(t)={} -\frac{\kappa^+(\rho^--\rho^+)}{2\pi} \text{P.V.}\int_{\TT} \frac{\tanh((g(\beta)-m(t))/2)(1+\tan^{2}((\underline{x}_t-\beta)/2)}{\tan^{2}((\underline{x}_t-\beta)/2)+\tanh^{2}((m(t)-g(\beta)/2))} \ d\beta \nonumber \\
&\hspace{1cm} - \frac{\mathcal{A}}{4\pi} \text{P.V.}\int_\TT\int_{\TT}\left[ \frac{\sin(w)}{\cosh(m(t)+1+h_2)-\cos(w)}
\right]  \nonumber \\
 & \times \frac{\sin(z)(g(\underline{x}_t-w-z)-g(\underline{x}_t-w))}{\left[ \cosh(h_{2}+g(\underline{x}_t-w-z)+1)-\cos(z)\right]\left[\cosh(h_{2}+g(\underline{x}_t-w)+1)-\cos(z)\right]}  dz \ dw \nonumber\\
  & \times \int_{0}^{1}\sinh\left[\lambda (g(\underline{x}_t-w-z)+1+h_{2})-(1-\lambda)(g(\underline{x}_t-w)+1+h_2) \right]\ d\lambda  \nonumber \\
 &:= L_{1}+L_{2}.
\end{align*}
Reasoning as before, we find that
\begin{align*}
L_1 &\leq\frac{\kappa^+(\rho^--\rho^+)}{2}\frac{m(t)}{\cosh((M(t)-m(t)/2))}
 \end{align*}
\begin{align*}
     L_2 &\leq \frac{8|\mathcal{A}|}{\pi (\cosh(h_2+1)-1)^{3}} \left(M(t)-m(t)\right)\sinh(1+M(t)-m(t)+h_2).
 \end{align*}
Thus, we have found the following estimate for $m'(t)$ 
\begin{multline}\label{ODE:m(t)}
-m'(t) \leq \frac{\kappa^{+}(\rho^2-\rho^1)}{2}\frac{m(t)}{\cosh((M(t)-m(t)/2))}\\
+\frac{8\mathcal{|A|}}{\pi (\cosh(h_2+1)-1)^{3}} \left(M(t)-m(t)\right)\sinh(1+M(t)-m(t)+h_2)\text{ a.e.}
\end{multline}
Next, we study the evolution of the oscillation  of $g$
$$
O(t)= M(t)-m(t).
$$
We combine \eqref{ODE:M(t)}-\eqref{ODE:m(t)} to show that
\begin{multline}\label{ODE:O}
    O'(t) \leq -\frac{\kappa^+(\rho^--\rho^+)}{2} O(t)\bigg{(} \frac{1}{\cosh(O(t)/2)}\\
    -\frac{16\mathcal{|A|}}{\kappa^+(\rho^--\rho^+)\pi (\cosh(h_2+1)-1)^{3}} \ \sinh(1+O(t)+h_2) \bigg{)}\text{ a.e.}
\end{multline}
By the smallness assumption \eqref{condition:smallness}, we have that at time $t=0$, 
\begin{multline*}
    O'(t)\bigg{|}_{t=0} \leq -\frac{\kappa^+(\rho^--\rho^+)}{2} O(0)\bigg{(} \frac{1}{\cosh(O(0)/2)}\\
    -\frac{16\mathcal{|A|}}{\kappa^+(\rho^--\rho^+)\pi (\cosh(h_2+1)-1)^{3}} \ \sinh(1+O(0)+h_2) \bigg{)}
\end{multline*}
and hence there exists a sufficiently small $0<\delta$ such that 
\begin{equation}
 O'(t)<0 \text{ a.e.}\Longrightarrow O(t)<O(0), \quad \forall 0\leq t<\delta.
\end{equation} 
We observe that, if the smallness condition \eqref{condition:smallness} holds at time $t=0$, then the smallness condition \eqref{condition:smallness} propagates with the evolution. Then, we find that \begin{equation}
 O(t)<O(0), \quad \forall t\geq 0.
\end{equation} 
Furthermore, by equation \eqref{ODE:O} we deduce that 
\begin{equation}\label{ODE:O:II}
    O'(t) \leq -\mu O(t)
\end{equation}
where 
$$
\mu= \frac{\kappa^+(\rho^--\rho^+)}{2}\bigg{(} \frac{1}{\cosh(O(0)/2)}\\
    -\frac{16\mathcal{|A|}}{\kappa^+(\rho^--\rho^+)\pi (\cosh(h_2+1)-1)^{3}} \ \sinh(1+O(0)+h_2) \bigg{)}
$$
which after a straightforward integration yields
\begin{equation}\label{ODE:O:decay}
    O(t) \leq O(0)\mbox{exp}(-\mu t).
\end{equation}
Once we have proved that the oscillation of $g$ decays, we only have to observe that 
$$
\max_x f(x,t)-\min_x f(x,t)=O(t)
$$
and that
$$
\max_x f(x,t)=\max_x f(x,t)-\min_x f(x,t)+\min_x f(x,t)\leq O(t)+\langle f(x,t)\rangle.
$$
From this inequality we conclude the result by noticing that the mean of $f$ is preserved by the evolution due to the divergence form of the Muskat problem \eqref{eq:divergence}.

\end{proof}

\section{Decay in $\dot{W}^{1,\infty}$}

In this section we show the decay in the $\dot{W}^{1,\infty}$ norm of the solution. 
\begin{proof}[Proof of Theorem \ref{teo2}]
The proof follows from a careful pointwise estimate as in Theorem \ref{teo1}. We observe that, due to the hypoteses in the statement and Theorem \ref{teo1}, the solution $f(x,t)$ remains non-negative for all times. We start rewriting the equation in a more convenient form.  Changing variables in the first integral of the evolution equation 
 \eqref{eq:I}, we observe that 
\begin{multline}\label{eq:I:change}
\p_{t} f(x)=\frac{\kappa^+(\rho^--\rho^+)}{4\pi} \text{P.V.}\int_\TT\frac{\sin(x-\beta)(\p_{x} f(x)-\p_{x} f(\beta))}{\cosh(f(x)-f(\beta))-\cos(x-\beta)} d\beta \\
+\frac{1}{4\pi} \text{P.V.}\int_\TT\frac{(\p_{x} f(x)\sinh(f(x)+h_2)+\sin(x-\beta))\varpi_2(\beta)}{\cosh(f(x)+h_2)-\cos(x-\beta)} \ d\beta
\end{multline}
where the second vorticity amplitude $\varpi_2(\beta)$ can be written as
\begin{equation*}
\varpi_2(\beta)=\mathcal{A}\text{P.V.}\int_\TT\frac{\sinh(h_2+f(\gamma))\p_{x}f(\gamma)}{\cosh(h_2+f(\gamma))-\cos(\beta-\gamma)}d\gamma.
\end{equation*}
Taking a space derivative of \eqref{eq:I:change} and writing the last term in \eqref{eq:I:change} as in 
\eqref{eq:simplified:1} we infer
\begin{align*}
\p_{t} \p_{x}f(x)&= I_{1}+I_{2}
\end{align*}
where
\begin{align*}
I_{1}&= \frac{\kappa^+(\rho^--\rho^+)}{4\pi} \text{P.V.} \ \p_{x}\int_\TT \frac{\sin(x-\beta)(\p_{x} f(x)-\p_{x} f(\beta))}{\cosh(f(x)-f(x-\beta))-\cos(x-\beta)} d\beta, \\
I_{2}&= \frac{\mathcal{A}}{4\pi} \text{P.V.} \  \p_{x} \int_\TT\int_{\TT}\frac{(\p_{x} f(x)\sinh(f(x)+h_2)+\sin(w))}{\cosh(f(x)+h_2)-\cos(w)}
 \frac{\sin(z)}{\cosh(h_{2}+f(x-w-z))-\cos(z)}dz \ dw.
\end{align*}
Using Leibniz rule we obtain 
\begin{multline*}
I_{1}=\frac{\kappa^+(\rho^--\rho^+)}{4\pi} \text{P.V.}\int_\TT \p_{x}\left(\frac{\sin(x-\beta)(\p_{x} f(x)-\p_{x} f(\beta))}{\cosh(f(x)-f(\beta))-\cos(x-\beta)} \right) d\beta \\
+ \frac{\kappa^+(\rho^--\rho^+)}{4\pi} \bigg(\displaystyle\lim_{\epsilon\to 0} \bigg[ \frac{\sin(x-\beta)(\p_{x} f(x)-\p_{x} f(\beta))}{\cosh(f(x)-f(\beta))-\cos(x-\beta)} \bigg]^{x+\pi}_{x+\epsilon} \\
 + \displaystyle\lim_{\epsilon\to 0} \bigg[ \frac{\sin(x-\beta)(\p_{x} f(x)-\p_{x} f(\beta))}{\cosh(f(x)-f(\beta))-\cos(x-\beta)} \bigg]^{x-\epsilon}_{x+\pi} \bigg)=I_{11}+I_{12}.
\end{multline*}
Using the periodicity, we find that the boundary terms vanish
\[ I_{12}= -\frac{\kappa^+(\rho^--\rho^+)}{4\pi} \left(0-\frac{\p_{x}f(x)}{\p_{x}f(x)^{2}+1}\right)+\frac{\kappa^+(\rho^--\rho^+)}{\pi} 
\left(0-\frac{\p_{x}f(x)}{\p_{x}f(x)^{2}+1}\right) =0. \]
Expanding the derivative in $I_{11}$ 
\begin{align*}
I_{11}=& \frac{\kappa^+(\rho^--\rho^+)}{4\pi}\text{P.V.}\int_\TT \frac{\cos(x-\beta)(\p_{x}f(x)-\p_{x}f(\beta))}{\cosh(f(x)-f(\beta))-\cos(x-\beta)} \ d\beta \\
&- \frac{\kappa^+(\rho^--\rho^+)}{4\pi}\text{P.V.}\int_\TT \frac{\sin(x-\beta)(\p_{x}f(x)-\p_{x}f(\beta))\left(\sinh(f(x)-f(\beta))\p_{x}f(x)+\sin(x-\beta)\right)}{(\cosh(f(x)-f(\beta))-\cos(x-\beta))^2} \ d\beta \\
&+ \frac{\kappa^+(\rho^--\rho^+)}{4\pi}\text{P.V.}\int_\TT \frac{\p_{x}^{2}f(x) \sin(x-\beta)}{\cosh(f(x)-f(\beta))-\cos(x-\beta)} \ d\beta= I_{111}+I_{112}+I_{113}.
\end{align*}
We can split further the second term as
\begin{multline*}
I_{112}= - \frac{\kappa^+(\rho^--\rho^+)}{4\pi}\text{P.V.}\int_\TT \frac{\sin(x-\beta)(\p_{x}f(x)-\p_{x}f(\beta))\sinh(f(x)-f(\beta))\p_{x}f(x)}{(\cosh(f(x)-f(\beta))-\cos(x-\beta))^2} \ d\beta \\
-\frac{\kappa^+(\rho^--\rho^+)}{4\pi}\text{P.V.}\int_\TT \frac{\sin^{2}(x-\beta)(\p_{x}f(x)-\p_{x}f(\beta))}{(\cosh(f(x)-f(\beta))-\cos(x-\beta))^2} \ d\beta 
\end{multline*}
from where, changing variables,
\begin{align*}
I_{111}+I_{112}&=  \frac{\kappa^+(\rho^--\rho^+)}{4\pi}\text{P.V.}\int_\TT \frac{\p_{x}f(x)-\p_{x}f(x-\beta)}{(\cosh(f(x)-f(x-\beta))-\cos(\beta))^2}  \\
&\quad \quad \times \bigg[\cosh(f(x)-f(x-\beta))\cos(\beta)-1-\sin(\beta)\sinh(f(x)-f(x-\beta))\p_{x}f(x)\bigg] d\beta.
\end{align*}
Hence
\begin{align}
I_{1}&= \frac{\kappa^+(\rho^--\rho^+)}{4\pi}\text{P.V.}\int_\TT \frac{\p_{x}f(x)-\p_{x}f(x-\beta)}{(\cosh(f(x)-f(x-\beta))-\cos(\beta))^2} \nonumber \\
&\quad \quad \times \bigg[\cosh(f(x)-f(x-\beta))\cos(\beta)-1-\sin(\beta)\sinh(f(x)-f(x-\beta))\p_{x}f(x)\bigg]  d\beta \nonumber \\ 
 &\quad + \frac{\kappa^+(\rho^--\rho^+)}{4\pi}\text{P.V.}\int_\TT \frac{\p_{x}^{2}f(x) \sin(x-\beta)}{\cosh(f(x)-f(\beta))-\cos(x-\beta)} \ d\beta. \label{I11}
\end{align}
On the other hand, computing the derivative in $I_{2}$ we find that
\[
I_{2}=I_{21}+I_{22}+I_{23}
\]
with
\begin{align*}
I_{21}&=\frac{\mathcal{A}}{4\pi} \text{P.V.} \int_\TT \int_\TT \frac{\left(\p_{x}^{2} f(x)\sinh(f(x)+h_{2})+(\p_{x}f(x))^{2}\cosh(f(x)+h_2)\right)}{\cosh(f(x)+h_{2})-\cos(w)} \\
& \hspace{2cm} \times  \frac{\sin(z)}{\cosh(h_{2}+f(x-w-z))-\cos(z)} \  dz \ dw \\
I_{22}&= -\frac{\mathcal{A}}{4\pi} \text{P.V.} \int_\TT \int_\TT \frac{(\p_{x} f(x))^{2}\sinh^{2}(f(x)+h_2)+\sin(w) \sinh(f(x)+h_2)\p_{x}f(x)}{(\cosh(f(x)+h_{2})-\cos(w))^2} \\
& \hspace{2cm} \times  \frac{\sin(z)}{\cosh(h_{2}+f(x-w-z))-\cos(z)} \  dz \ dw \\
I_{23}&= -\frac{\mathcal{A}}{4\pi} \text{P.V.} \int_\TT \int_\TT \frac{(\p_{x}f(x)\sinh(f(x)+h_2)+\sin(w))(\sinh(h_{2}+f(x-w-z))\p_{x}f(x-w-z))}{\cosh(f(x)+h_2)-\cos(w)} \\
& \hspace{2cm} \times  \frac{\sin(z)}{(\cosh(h_{2}+f(x-w-z))-\cos(z))^{2}} \  dz \ dw
\end{align*}

Collecting the previous computations, we have shown that 
\begin{align}
\p_{t}\p_{x}f(x)&= \frac{\kappa^+(\rho^--\rho^+)}{4\pi}\text{P.V.}\int_\TT \frac{\p_{x}f(x)-\p_{x}f(x-\beta)}{(\cosh(f(x)-f(x-\beta))-\cos(\beta))^2} \nonumber \\
&\quad \quad \times \bigg[\cosh(f(x)-f(x-\beta))\cos(\beta)-1-\sin(\beta)\sinh(f(x)-f(x-\beta))\p_{x}f(x)\bigg]d\beta \nonumber \\ 
 &\quad + \frac{\kappa^+(\rho^--\rho^+)}{4\pi}\text{P.V.}\int_\TT \frac{\p_{x}^{2}f(x) \sin(x-\beta)}{\cosh(f(x)-f(\beta))-\cos(x-\beta)} \ d\beta  \nonumber \\
 &+ \frac{\mathcal{A}}{4\pi} \text{P.V.} \int_\TT \int_\TT \frac{\left(\p_{x}^{2} f(x)\sinh(f(x)+h_{2})+(\p_{x}f(x))^{2}\cosh(f(x)+h_2)\right)}{\cosh(f(x)+h_{2})-\cos(w)}  \nonumber \\
& \hspace{2cm} \times  \frac{\sin(z)}{\cosh(h_{2}+f(x-w-z))-\cos(z)} \  dz \ dw  \nonumber \\
& -\frac{\mathcal{A}}{4\pi} \text{P.V.} \int_\TT \int_\TT \frac{(\p_{x} f(x))^{2}\sinh^{2}(f(x)+h_2)+\sin(w) \sinh(f(x)+h_2)\p_{x}f(x)}{(\cosh(f(x)+h_{2})-\cos(w))^2} \nonumber  \\
& \hspace{2cm} \times  \frac{\sin(z)}{\cosh(h_{2}+f(x-w-z))-\cos(z)} \  dz \ dw \nonumber  \\
& -\frac{\mathcal{A}}{4\pi} \text{P.V.} \int_\TT \int_\TT \frac{(\p_{x}f(x)\sinh(f(x)+h_2)+\sin(w))(\sinh(h_{2}+f(x-w-z))\p_{x}f(x-w-z))}{\cosh(f(x)+h_2)-\cos(w)} \nonumber  \\
& \hspace{2cm} \times  \frac{\sin(z)}{(\cosh(h_{2}+f(x-w-z))-\cos(z))^{2}} \  dz \ dw. \label{evol:eq}
\end{align}
As in the proof of Theorem \ref{teo1}, we define
\begin{equation}\label{max:der:g}
M(t)=\displaystyle\max_{x} \p_{x}f(x,t)=\p_{x}f(\overline{x}_t,t)>0
\end{equation}
and 
\begin{equation}\label{min:der:g}
m(t)=\displaystyle\min_{x} \p_{x}f(x,t)=\p_{x}f(\underline{x}_t,t)<0.
\end{equation}
Due to the smoothness of $f$, we infer by the Sobolev embedding that $\p_{x}f\in C^1([0,T]\times\TT)$. As a consequence, following the arguments in Section \ref{decay:Linf}, we find that $M(t)$ and $m(t)$ are Lipschitz functions and then, due to Radamacher's Theorem, they are differentiable almost everywhere. Moreover, 
\[ M'(t)=\p_{t}\p_{x}f(\overline{x}_{t},t), \quad m'(t)=\p_{t}\p_{x}f(\underline{x}_{t},t)\quad \text{a.e.}.\]
Therefore, evaluating \eqref{evol:eq} at $x=\overline{x}_{t}$ and noticing that $\p_{x}^{2}f(\overline{x}_{t})=0$ we find that
\begin{equation}
M'(t)=J_{1}+J_{2}+J_{3}+J_{4}
\end{equation}
where
\begin{align*}
J_{1}&=\frac{\kappa^+(\rho^--\rho^+)}{4\pi}\text{P.V.}\int_\TT \frac{M(t)-\p_{x}f(\overline{x}_{t}-\beta)}{(\cosh(f(\overline{x}_{t})-f(\overline{x}_{t}-\beta))-\cos(\beta))^2}\nonumber  \\
&\quad \quad \times \bigg[\cosh(f(\overline{x}_{t})-f(\overline{x}_{t}-\beta))\cos(\beta)-1-\sin(\beta)\sinh(f(\overline{x}_{t})-f(\overline{x}_{t}-\beta))M(t)\bigg] \ d \beta, \\
J_{2}&= \frac{\mathcal{A}}{4\pi} \text{P.V.} \int_\TT \int_\TT \frac{M(t)^{2}\cosh(f(\overline{x}_{t})+h_2)}{\cosh(f(\overline{x}_{t})+h_{2})-\cos(w)}  \times  \frac{\sin(z)}{\cosh(h_{2}+f(\overline{x}_{t}-w-z))-\cos(z)} \  dz \ dw \nonumber \\
J_{3}&=-\frac{\mathcal{A}}{4\pi} \text{P.V.} \int_\TT \int_\TT \frac{M(t)^{2}\sinh^{2}(f(\overline{x}_{t})+h_2)+\sin(w) \sinh(f(\overline{x}_{t})+h_2)M(t)}{(\cosh(f(\overline{x}_{t})+h_{2})-\cos(w))^2} \nonumber  \\
& \hspace{2cm} \times  \frac{\sin(z)}{\cosh(h_{2}+f(\overline{x}_{t}-w-z))-\cos(z)} \  dz \ dw \\
J_{4}&= -\frac{\mathcal{A}}{4\pi} \text{P.V.} \int_\TT \int_\TT \frac{(M(t)\sinh(f(\overline{x}_{t})+h_2)+\sin(w))(\sinh(h_{2}+f(\overline{x}_{t}-w-z))\p_{x}f(\overline{x}_{t}-w-z))}{\cosh(f(\overline{x}_{t})+h_2)-\cos(w)} \nonumber  \\
& \hspace{2cm} \times  \frac{\sin(z)}{(\cosh(h_{2}+f(\overline{x}_{t}-w-z))-\cos(z))^{2}} \  dz \ dw
\end{align*}
%
%and
%
%\begin{align*}
%J_{5}&=-\frac{\mathcal{A}}{2\pi} \text{P.V.} \int_\TT \int_\TT \frac{(M(t)\sinh(g(\overline{x}_{t})+1+h_2)+\sin(w))(\sinh(h_{2}+1+g(\overline{x}_{t}-w-z))\p_{x}g(\overline{x}_{t}-w-z))}{\cosh(g(\overline{x}_{t})+1+h_2)-\cos(w)} \nonumber \\
%&\times \frac{\sin(z)\cos(z)(g(\overline{x}_{t}-w-z)-g(\overline{x}_{t}-w))}{(\cosh(h_{2}+1+g(\overline{x}_{t}-w))-\cos(z))^{2}(\cosh(h_{2}+1+g(\overline{x}_{t}-w-z))-\cos(z))^{2}} \nonumber \\
% &\hspace{0.2cm} \times\int_{0}^{1}\sinh\left[\lambda (g(\overline{x}_{t}-w-z)+1+h_{2})-(1-\lambda)(g(\overline{x}_{t}-w)+1+h_2) \right]\ d\lambda \  dz \ dw.
%\end{align*}

Using that the solution $f$ remains non-neagative due to Theorem \ref{teo1}, we can obtain the following upper bounds for $J_{2}$ to $J_{4}$,
\begin{align}\label{J2est}
\abs{J_{2}}\leq \pi\abs{\mathcal{A}} \frac{\cosh(\norm{f}_{L^{\infty}}+h_{2})}{(\cosh(h_{2})-1)^2}M(t)^2  
\end{align}
\begin{align}\label{J3est}
\abs{J_{3}}\leq  \pi \abs{\mathcal{A}}\frac{\sinh^{2}(\norm{f}_{L^{\infty}}+h_2)}{(\cosh(h_{2})-1)^{3}}M(t)^{2}+\pi\abs{\mathcal{A}}\frac{\sinh(\norm{f}_{L^{\infty}}+h_{2})}{(\cosh(h_{2})-1)^3}M(t)
\end{align}
and
\begin{align}\label{J4est}
\abs{J_{4}}\leq   \pi \abs{\mathcal{A}}\frac{\sinh^{2}(\norm{f}_{L^{\infty}}+h_2)}{(\cosh(h_{2})-1)^{3}}M(t)\|\partial_x f\|_{L^\infty}+ \pi\abs{\mathcal{A}}\frac{\sinh(\norm{f}_{L^{\infty}}+h_{2})}{(\cosh(h_{2})-1)^3}\|\partial_x f\|_{L^\infty}.
\end{align}

Manipulating $J_{1}$ and using the trigonometric identities 
 \[2 \sin^{2}(x/2)= 1-\cos(x),\qquad 2\sinh^{2}(x/2)=\cosh(x)-1,\]
 we observe that
\begin{align*}
J_{1}&= \frac{\kappa^+(\rho^--\rho^+)}{4\pi}\text{P.V.}\int_\TT \frac{M(t)-\p_{x}f(\overline{x}_{t}-\beta)}{(\cosh(f(\overline{x}_{t})-f(\overline{x}_{t}-\beta))-\cos(\beta))^2}\nonumber \\
&\quad  \times \bigg[(\cos(\beta)-1) + \left(\cosh(f(\overline{x}_{t})-f(\overline{x}_{t}-\beta))-1\right)\cos(\beta) -\sin(\beta)\sinh(f(\overline{x}_{t})-f(\overline{x}_{t}-\beta))M(t)\bigg] \ d \beta \\
&= J_{11} + J_{12}
\end{align*}
with 
$$
J_{11}=-  \frac{\kappa^+(\rho^--\rho^+)}{8\pi}\text{P.V.}\int_\TT \frac{(M(t)-\p_{x}f(\overline{x}_{t}-\beta))}{\sin^{2}(\beta/2)} \frac{1}{\left(1+\frac{\sinh^{2}((f(\overline{x}_{t})-f(\overline{x}_{t}-\beta))/2)}{\sin^{2}(\beta/2)}\right)^2}d\beta ,
$$
and
\begin{align*}
J_{12}&=\frac{\kappa^+(\rho^--\rho^+)}{4\pi}\text{P.V.}\int_\TT \frac{M(t)-\p_{x}f(\overline{x}_{t}-\beta)}{(\cosh(f(\overline{x}_{t})-f(\overline{x}_{t}-\beta))-\cos(\beta))^2}\nonumber \\
&\quad  \times \bigg[ \left(\cosh(f(\overline{x}_{t})-f(\overline{x}_{t}-\beta))-1\right)\cos(\beta) -\sin(\beta)\sinh(f(\overline{x}_{t})-f(\overline{x}_{t}-\beta))M(t)\bigg]\ d \beta. \\
\end{align*}
In order to estimate $J_{11}$ we observe that the trigonometric analog of the incremental quotient can be estimated as
\begin{align}
\frac{\sinh((f(\overline{x}_{t})-f(\overline{x}_{t}-\beta))/2)}{\sin(\beta/2)}&=\frac{\sinh((f(\overline{x}_{t})-f(\overline{x}_{t}-\beta))/2)}{\beta/2}\frac{\beta/2}{\sin(\beta/2)} \nonumber \\
&\leq \frac{\sinh(\|\partial_x f\|_{L^\infty}(\beta/2))}{\beta/2}\frac{\pi}{2}  \nonumber \\
&\leq \sinh(\|\partial_x f\|_{L^\infty}(\pi/2)) \label{est:quo},
\end{align}
and then, for sufficiently small initial data, it can be assumed to be arbitrarily small. Furthermore,  using the geometric series 
$$
\frac{1}{1+r}=\sum_{n=0}^\infty r^n(-1)^n
$$
and its derivative
$$
\frac{1}{(1+r)^2}=\sum_{k=0}^\infty (1+k) r^{k}(-1)^k
$$
we find the expression
\begin{align}
 \frac{1}{\left(1+\frac{\sinh^{2}((f(\overline{x}_{t})-f(\overline{x}_{t}-\beta))/2)}{\sin^{2}(\beta/2)}\right)^2}&= 1+ \displaystyle\sum_{k=1}^{\infty} \left(\frac{\sinh((f(\overline{x}_{t})-f(\overline{x}_{t}-\beta))/2)}{\sin(\beta/2)}\right)^{2k}(-1)^{k}(1+k) \label{expre:sum}
\end{align}
where the convergence of the series is ensured (at least locally in time) if the Lipschitz seminorm of the initial data is small enough. Combining \eqref{est:quo} and \eqref{expre:sum} we have that
\begin{align}
J_{11}& \leq -  \frac{\kappa^+(\rho^--\rho^+)}{8\pi}\bigg{[}\text{P.V.}\int_\TT \frac{(M(t)-\p_{x}f(\overline{x}_{t}-\beta))}{\sin^{2}(\beta/2)} \ d \beta\nonumber\\ &\quad +\text{P.V.}\int_\TT\sum_{k=1}^{\infty} \frac{(M(t)-\p_{x}f(\overline{x}_{t}-\beta))}{\sin^{2}(\beta/2)}  \left(\frac{\sinh((f(\overline{x}_{t})-f(\overline{x}_{t}-\beta))/2)}{\sin(\beta/2)}\right)^{2k}(1+k) \ d \beta\bigg{]}\nonumber\\
& \leq -  \frac{\kappa^+(\rho^--\rho^+)}{8\pi}\bigg{[}\text{P.V.}\int_\TT \frac{(M(t)-\p_{x}f(\overline{x}_{t}-\beta))}{\sin^{2}(\beta/2)} \ d \beta\nonumber\\ &\quad +\left[\frac{1}{\left(1-\left(\sinh(\|\partial_x f\|_{L^\infty}(\pi/2))\right)^{2}\right)^2}-1\right]\text{P.V.}\int_\TT \frac{(M(t)-\p_{x}f(\overline{x}_{t}-\beta))}{\sin^{2}(\beta/2)}  \ d \beta\bigg{]} \label{J1est}
\end{align}
To bound the term $J_{12}$ we notice that
\begin{align*}
\cosh(f(\overline{x}_{t})-f(\overline{x}_{t}-\beta))-1&=\int_0^1 \partial_\lambda\cosh(\lambda(f(\overline{x}_{t})-f(\overline{x}_{t}-\beta)))d\lambda \leq |\beta| \|\partial_x f\|_{L^\infty} \sinh(\beta \|\partial_x f\|_{L^\infty}) \\
\sinh(f(\overline{x}_{t})-f(\overline{x}_{t}-\beta))&=\int_0^1 \partial_\lambda\sinh(\lambda(f(\overline{x}_{t})-f(\overline{x}_{t}-\beta)))d\lambda\leq |\beta|\|\partial_x f\|_{L^\infty} \cosh(\beta \|\partial_x f\|_{L^\infty}).
\end{align*}
and that
$$
\frac{\sin(x)}{\sin(x/2)}\leq 2,\quad \frac{\sinh(x \|\partial_x f\|_{L^\infty})}{\sin(x/2)}\leq \sinh(\pi \|\partial_x f\|_{L^\infty}),\quad \frac{x}{\sin(x/2)}\leq \pi.
$$
Therefore, we have that
\begin{align}
\abs{J_{12}} \leq& \frac{\kappa^+(\rho^--\rho^+)}{16\pi}\|\partial_x f\|_{L^\infty}(\cosh(\pi \|\partial_x f\|_{L^\infty}) 2\pi+\pi \sinh(\pi \|\partial_x f\|_{L^\infty}))\nonumber\\
&\times
\text{P.V.}\int_\TT \frac{M(t)-\p_{x}f(\overline{x}_{t}-\beta)}{\sin^{2}(\beta/2)} d\beta. \label{J12est}
\end{align}
As a consequence, collecting estimates \eqref{J2est},  \eqref{J3est}, \eqref{J4est}, \eqref{J1est} and \eqref{J12est} we have shown that 
\begin{align*}
M'(t) &\leq -  \frac{\kappa^+(\rho^--\rho^+)}{8\pi}\text{P.V.}\int_\TT \frac{(M(t)-\p_{x}f(\overline{x}_{t}-\beta))}{\sin^{2}(\beta/2)} \ d \beta \left(1-\mathcal{P}(\|\partial_x f\|_{L^\infty})\right)\\
&\quad+ \pi\abs{\mathcal{A}} \frac{\cosh(\norm{f}_{L^{\infty}}+h_{2})}{(\cosh(h_{2})-1)^2}M(t)^2 \\
&\quad + 2\pi \abs{\mathcal{A}}\frac{\sinh^{2}(\norm{f}_{L^{\infty}}+h_2)}{(\cosh(h_{2})-1)^{3}}M(t)\|\partial_x f\|_{L^\infty}+2\pi\abs{\mathcal{A}}\frac{\sinh(\norm{f}_{L^{\infty}}+h_{2})}{(\cosh(h_{2})-1)^3}\|\partial_x f\|_{L^\infty} \text{ a.e.},
\end{align*}
with
$$
\mathcal{P}=\left[\frac{1}{\left(1-\left(\sinh(\|\partial_x f\|_{L^\infty}(\pi/2))\right)^{2}\right)^2}-1\right]+\frac{1}{2}\|\partial_x f\|_{L^\infty}(\cosh(\pi \|\partial_x f\|_{L^\infty}) 2\pi+\pi \sinh(\pi \|\partial_x f\|_{L^\infty})).
$$
Using that
$$
\text{P.V.}\int_\TT \frac{(M(t)-\p_{x}f(\overline{x}_{t}-\beta))}{\sin^{2}(\beta/2)} \ d \beta\geq 2\pi M(t)
$$
we conclude that
\begin{align*}
M'(t) &\leq -  \frac{\kappa^+(\rho^--\rho^+)}{4}M(t) \left(1-\mathcal{P}(\|\partial_x f\|_{L^\infty})\right)+ \pi\abs{\mathcal{A}} \frac{\cosh(\norm{f}_{L^{\infty}}+h_{2})}{(\cosh(h_{2})-1)^2}M(t)^2 \\
&\quad + 2\pi \abs{\mathcal{A}}\frac{\sinh^{2}(\norm{f}_{L^{\infty}}+h_2)}{(\cosh(h_{2})-1)^{3}}M(t)\|\partial_x f\|_{L^\infty}+2\pi\abs{\mathcal{A}}\frac{\sinh(\norm{f}_{L^{\infty}}+h_{2})}{(\cosh(h_{2})-1)^3}\|\partial_x f\|_{L^\infty} \text{ a.e.},
\end{align*}
Similarly, when we evaluate \eqref{evol:eq} in $x=\underline{x}_{t}$ and use that $\p_{x}^{2}f(\underline{x}_{t})=0$, we find that
\begin{align*}
-m'(t) &\leq   \frac{\kappa^+(\rho^--\rho^+)}{8\pi}\text{P.V.}\int_\TT \frac{(m(t)-\p_{x}f(\overline{x}_{t}-\beta))}{\sin^{2}(\beta/2)} \ d \beta \left(1-\mathcal{P}(\|\partial_x f\|_{L^\infty})\right)\\
&\quad+ \pi\abs{\mathcal{A}} \frac{\cosh(\norm{f}_{L^{\infty}}+h_{2})}{(\cosh(h_{2})-1)^2}m(t)^2 \\
&\quad + 2\pi \abs{\mathcal{A}}\frac{\sinh^{2}(\norm{f}_{L^{\infty}}+h_2)}{(\cosh(h_{2})-1)^{3}}|m(t)|\norm{\partial_x f}_{L^{\infty}}+2\pi\abs{\mathcal{A}}\frac{\sinh(\norm{f}_{L^{\infty}}+h_{2})}{(\cosh(h_{2})-1)^3}\norm{\partial_x f}_{L^{\infty}} \text{ a.e.},
\end{align*}
As before, we estimate that
$$
-\text{P.V.}\int_\TT \frac{(m(t)-\p_{x}f(\overline{x}_{t}-\beta))}{\sin^{2}(\beta/2)} \ d \beta\geq -2\pi m(t),
$$
thus,
\begin{align*}
-m'(t) &\leq   \frac{\kappa^+(\rho^--\rho^+)}{4}m(t) \left(1-\mathcal{P}(\|\partial_x f\|_{L^\infty})\right)+ \pi\abs{\mathcal{A}} \frac{\cosh(\norm{f}_{L^{\infty}}+h_{2})}{(\cosh(h_{2})-1)^2}m(t)^2 \\
&\quad + 2\pi \abs{\mathcal{A}}\frac{\sinh^{2}(\norm{f}_{L^{\infty}}+h_2)}{(\cosh(h_{2})-1)^{3}}|m(t)|\norm{\partial_x f}_{L^{\infty}}+2\pi\abs{\mathcal{A}}\frac{\sinh(\norm{f}_{L^{\infty}}+h_{2})}{(\cosh(h_{2})-1)^3}\norm{\partial_x f}_{L^{\infty}} \text{ a.e.},
\end{align*}
We observe that when 
$$
\|\partial_x f\|_{L^\infty}=M(t),
$$
we find that
\begin{align}
\frac{d}{dt}\|\partial_x f\|_{L^\infty} &\leq -  \frac{\kappa^+(\rho^--\rho^+)}{4}\|\partial_x f\|_{L^\infty} \left(1-\mathcal{P}(\|\partial_x f\|_{L^\infty})\right)+ \pi\abs{\mathcal{A}} \frac{\cosh(\norm{f}_{L^{\infty}}+h_{2})}{(\cosh(h_{2})-1)^2}\|\partial_x f\|_{L^\infty}^2 \nonumber\\
&\quad + 2\pi \abs{\mathcal{A}}\frac{\sinh^{2}(\norm{f}_{L^{\infty}}+h_2)}{(\cosh(h_{2})-1)^{3}}\|\partial_x f\|_{L^\infty}^2+2\pi\abs{\mathcal{A}}\frac{\sinh(\norm{f}_{L^{\infty}}+h_{2})}{(\cosh(h_{2})-1)^3}\|\partial_x f\|_{L^\infty} \text{ a.e.}. \label{eq:norminf}
\end{align}
In the case where
$$
\|\partial_x f\|_{L^\infty}=-m(t),
$$
we also obtain the previous inequality \eqref{eq:norminf}. As a consequence, we conclude that \eqref{eq:norminf} holds almost everywhere in time and, if the initial data is small enough according to the statement of the Theorem (see equation \eqref{eq:hipnorminf}), we conclude the exponential decay towards the equilibrium.

\end{proof}

\section{Global existence of weak solutions}

In this section we prove the global existence of weak solutions for \eqref{eq:I}. The proof follows the same approach as the proof in \cite{gancedo2020global2}: 
\begin{enumerate}
\item We adopt a vanishing viscosity approach. This regularization procedure leads to globally defined solutions and is well-adapted to our pointwise estimates obtained in the previous sections. 
\item Then, we repeat the pointwise estimates in the previous sections to conclude that the approximate solutions decay in the Lipschitz norm.
\item We perform $L^\infty(0,T;H^1)$ estimates. These estimates give us that the approximate solutions are uniformly bounded in $L^2(0,T;H^{3/2})$. This parabolic gain of regularity ensures the compacity and allow us to pass to the limit.
\item This limit is a weak solution of \eqref{eq:I}.
\end{enumerate}

\begin{proof}[Proof of Theorem \ref{teo3}] 

\textbf{Step 1 (The regularization approach):} We consider the following regularized problem
\begin{multline}\label{eq:Ireg}
\p_{t} f^\varepsilon(x)=\frac{\kappa^+(\rho^--\rho^+)}{4\pi} \text{P.V.}\int_\TT\frac{\sin(\beta)(\p_{x} f^\varepsilon(x)-\p_{x} f^\varepsilon(x-\beta))}{\cosh(f^\varepsilon(x)-f^\varepsilon(x-\beta))-\cos(\beta)} d\beta\\
+\frac{1}{4\pi} \text{P.V.}\int_\TT\frac{(\p_{x} f^\varepsilon(x)\sinh(f^\varepsilon(x)+h_2)+\sin(x-\beta))\varpi_2(\beta)}{\cosh(f^\varepsilon(x)+h_2)-\cos(x-\beta)}d\beta+\varepsilon \partial_x^2 f^\varepsilon(x),
\end{multline}
with initial data
$$
f_0^\varepsilon(x)=\mathcal{J}_{\varepsilon}* f_0
$$
where $\mathcal{J}_{\varepsilon}$ denotes the periodic heat kernel at time $t=\varepsilon$. From this point onwards, abusing notation we drop the $\varepsilon$ notation from $f$. As before, equation \eqref{eq:Ireg} can be written in divergence form
\begin{multline}\label{eq:divergencereg}
\p_{t} f^\varepsilon(x)=\frac{\kappa^+(\rho^--\rho^+)}{\pi} \p_{x}\text{P.V.}\int_\TT \arctan\left( \frac{\tanh((f^\varepsilon(x)-f^\varepsilon(x-\beta))/2)}{\tan(\beta/2)}\right) \ d\beta\\
+\frac{1}{4\pi} \p_{x}\text{P.V.}\int_\TT\log(\cosh(f^\varepsilon(x)+h_2)-\cos(x-\beta))\varpi_2(\beta)d\beta+\varepsilon \partial_x^2 f^\varepsilon(x),
\end{multline}

The local well-posedness of \eqref{eq:Ireg} in $C([0,T_\varepsilon],H^3)$ is obtained as in \cite[\S 3.1]{berselli2014local}. Furthermore, this sequence of approximate solutions exists globally. Indeed, multiplying the previous equation by $f$ and integrating by parts, we find that
\begin{multline*}
\frac{1}{2}\frac{d}{dt}\|f\|_{L^2}^2+\varepsilon\|\partial_x f\|_{L^2}^2\\
=-\frac{\kappa^+(\rho^--\rho^+)}{\pi} \int_\TT \p_{x}f\text{P.V.}\int_\TT \arctan\left( \frac{\tanh((f(x)-f(x-\beta))/2)}{\tan(\beta/2)}\right) \ d\beta \ dx\\
-\frac{1}{4\pi} \int_\TT \p_{x}f\text{P.V.}\int_\TT\log(\cosh(f(x)+h_2)-\cos(x-\beta))\varpi_2(\beta)d\beta \ dx\\
\leq \frac{\varepsilon}{2}\|\partial_x f\|_{L^2}^2 + C_\varepsilon(\|f_0\|_{W^{1,\infty}}).
\end{multline*}
If we now multiply \eqref{eq:Ireg} by $-\partial_x^2 f$ and integrate in space we find
\begin{multline*}
\frac{1}{2}\frac{d}{dt}\|\partial_x f\|_{L^2}^2+\varepsilon\|\partial_x^2 f\|_{L^2}^2\\
=-\frac{\kappa^+(\rho^--\rho^+)}{\pi} \int_\TT \p_{x}^2f\partial_x\text{P.V.}\int_\TT \arctan\left( \frac{\tanh((f(x)-f(x-\beta))/2)}{\tan(\beta/2)}\right) \ d\beta \ dx\\
-\frac{1}{4\pi} \int_\TT \p_{x}^2f \partial_x \text{P.V.}\int_\TT\log(\cosh(f(x)+h_2)-\cos(x-\beta))\varpi_2(\beta)d\beta \  dx.
\end{multline*}
After taking $\varepsilon>0$ small enough, we can ensure that $f_0^\varepsilon$ also satisfies the hypotheses in Theorems \ref{teo1} and \ref{teo2}. Then, using that
\begin{align*}
\partial_x\text{P.V.}\int_\TT \arctan\left( \frac{\tanh((f(x)-f(x-\beta))/2)}{\tan(\beta/2)}\right) \ d\beta&=\text{P.V.}\int_\TT\frac{\sin(\beta)(\p_{x} f(x)-\p_{x} f(x-\beta))}{\cosh(f(x)-f(x-\beta))-\cos(\beta)}d\beta\\
&=\text{P.V.}\int_\TT\frac{\sin(\beta)}{\sin(\beta/2)}\frac{(\p_{x} f(x)-\p_{x} f(x-\beta))}{\sin(\beta/2)}\\
&\quad\times\frac{1}{2\left(1+\frac{\sinh^2((f(x)-f(x-\beta))/2)}{\sin^2(\beta/2)}\right)}d\beta
\end{align*}
together with Jensen inequality, we obtain
$$
\|\partial_x\text{P.V.}\int_\TT \arctan\left( \frac{\tanh((f(x)-f(x-\beta))/2)}{\tan(\beta/2)}\right) \ d\beta\|_{L^2}\leq C(\|f_0\|_{W^{1,\infty}})\|f\|_{\dot{H}^{3/2}},
$$
so that, using Sobolev interpolation
$$
\|f\|_{\dot{H}^{3/2}}\leq \|\p_{x}f\|^{1/2}_{L^{2}}\|\p^{2}_{x}f\|^{1/2}_{L^{2}}
$$ 
and Young's inequality
$$
a^{3/2}b^{1/2}\leq \varepsilon a^2+C_\varepsilon b^2,
$$
we conclude
$$
\frac{1}{2}\frac{d}{dt}\|\partial_x f\|_{L^2}^2+\varepsilon\|\partial_x^2 f\|_{L^2}^2\leq \frac{\varepsilon}{2}\|\partial_x^2 f\|_{L^2}^2 + C_\varepsilon(\|f_0\|_{W^{1,\infty}})\|\partial_x f\|_{L^2}^2.
$$
Using the same ideas (see also \cite{gancedo2017survey,granero2014global}), we conclude that
$$
\frac{1}{2}\frac{d}{dt}\|\partial_x^3 f\|_{L^2}^2+\varepsilon\|\partial_x^4 f\|_{L^2}^2\leq \frac{\varepsilon}{2}\|\partial_x^4 f\|_{L^2}^2 + C_\varepsilon(\|f_0\|_{W^{1,\infty}})\|\partial_x^3 f\|_{L^2}^2.
$$

\textbf{Step 2 (Uniform estimates in $W^{1,\infty}$):} The approximate solutions constructed above satisfy the pointwise estimates obtained in the previous sections. In particular, we have the global bounds
$$
\|f^\varepsilon\|_{L^\infty}\leq \|f_0\|_{L^\infty}
$$
and
$$
\|\partial_x f^\varepsilon\|_{L^\infty}\leq \|\partial_x f_0\|_{L^\infty}.
$$

\textbf{Step 3 (Uniform estimates in $H^{1}$):} We recall \eqref{evol:eq}. Then we find that
$$
\frac{1}{2}\frac{d}{dt}\|\partial_x f\|_{L^2}^2+\varepsilon\|\partial_x^2 f\|_{L^2}^2=I_1+I_2+I_3+I_4+I_5
$$
with
\begin{align*}
I_1&= \frac{\kappa^+(\rho^--\rho^+)}{4\pi}\text{P.V.}\int_\TT\int_\TT\partial_x f(x) \frac{\p_{x}f(x)-\p_{x}f(x-\beta)}{(\cosh(f(x)-f(x-\beta))-\cos(\beta))^2} \nonumber \\
&\quad \quad \times \bigg[\cosh(f(x)-f(x-\beta))\cos(\beta)-1-\sin(\beta)\sinh(f(x)-f(x-\beta))\p_{x}f(x)\bigg]d\beta \ dx \nonumber \\ 
I_2&=\frac{\kappa^+(\rho^--\rho^+)}{4\pi}\text{P.V.}\int_\TT\int_\TT\partial_x f(x) \frac{\p_{x}^{2}f(x) \sin(x-\beta)}{\cosh(f(x)-f(\beta))-\cos(x-\beta)} \ d\beta \ dx  \nonumber \\
I_3&= \frac{\mathcal{A}}{4\pi} \text{P.V.} \int_\TT \int_\TT\int_\TT\partial_x f(x) \frac{\left(\p_{x}^{2} f(x)\sinh(f(x)+h_{2})+(\p_{x}f(x))^{2}\cosh(f(x)+h_2)\right)}{\cosh(f(x)+h_{2})-\cos(w)}  \nonumber \\
& \hspace{2cm} \times  \frac{\sin(z)}{\cosh(h_{2}+f(x-w-z))-\cos(z)} \  dz \ dw \ dx  \nonumber \\
I_4&= -\frac{\mathcal{A}}{4\pi} \text{P.V.} \int_\TT \int_\TT \int_\TT\partial_x f(x) \frac{(\p_{x} f(x))^{2}\sinh^{2}(f(x)+h_2)+\sin(w) \sinh(f(x)+h_2)\p_{x}f(x)}{(\cosh(f(x)+h_{2})-\cos(w))^2} \nonumber  \\
& \hspace{2cm} \times  \frac{\sin(z)}{\cosh(h_{2}+f(x-w-z))-\cos(z)} \  dz \ dw \ dx\nonumber  \\
I_5&= -\frac{\mathcal{A}}{4\pi} \text{P.V.} \int_\TT \int_\TT \int_\TT\partial_x f(x) \frac{(\p_{x}f(x)\sinh(f(x)+h_2)+\sin(w))(\sinh(h_{2}+f(x-w-z))\p_{x}f(x-w-z))}{\cosh(f(x)+h_2)-\cos(w)} \nonumber  \\
& \hspace{2cm} \times  \frac{\sin(z)}{(\cosh(h_{2}+f(x-w-z))-\cos(z))^{2}} \  dz \ dw \ dx. 
\end{align*}
We write $I_1$ as follows
$$
I_1=I_{11}+I_{12}
$$
with
$$
I_{11}=-\frac{\kappa^+(\rho^--\rho^+)}{8\pi}\int_\TT\text{P.V.}\int_\TT\partial_x f(x) \frac{(\partial_x f(x)-\p_{x}f(x-\beta))}{\sin^{2}(\beta/2)} \frac{1}{\left(1+\frac{\sinh^{2}((f(x)-f(x-\beta))/2)}{\sin^{2}(\beta/2)}\right)^2}d\beta \ dx,
$$
and
\begin{align*}
I_{12}&=\frac{\kappa^+(\rho^--\rho^+)}{4\pi}\int_\TT\text{P.V.}\int_\TT \partial_x f(x)\frac{\partial_x f(x)-\p_{x}f(x-\beta)}{(\cosh(f(x)-f(x-\beta))-\cos(\beta))^2}\nonumber \\
&\quad  \times \bigg[ \left(\cosh(f(x)-f(x-\beta))-1\right)\cos(\beta) -\sin(\beta)\sinh(f(x)-f(x-\beta))\partial_x f(x)\bigg]\ d \beta \ dx.
\end{align*}
Changing variables we find that
\begin{align*}
I_{11}&=-\frac{\kappa^+(\rho^--\rho^+)}{8\pi}\int_\TT\text{P.V.}\int_\TT\partial_x f(x) \frac{(\partial_x f(x)-\p_{x}f(\beta))}{\sin^{2}((x-\beta)/2)} \frac{1}{\left(1+\frac{\sinh^{2}((f(x)-f(\beta))/2)}{\sin^{2}((x-\beta)/2)}\right)^2}d\beta \ dx\\
&=\frac{\kappa^+(\rho^--\rho^+)}{8\pi}\int_\TT\text{P.V.}\int_\TT\partial_x f(\beta) \frac{(\partial_x f(x)-\p_{x}f(\beta))}{\sin^{2}((x-\beta)/2)} \frac{1}{\left(1+\frac{\sinh^{2}((f(x)-f(\beta))/2)}{\sin^{2}((x-\beta)/2)}\right)^2}d\beta \ dx\\
&=-\frac{\kappa^+(\rho^--\rho^+)}{16\pi}\int_\TT\text{P.V.}\int_\TT \frac{(\partial_x f(x)-\p_{x}f(\beta))^2}{\sin^{2}((x-\beta)/2)} \frac{1}{\left(1+\frac{\sinh^{2}((f(x)-f(\beta))/2)}{\sin^{2}((x-\beta)/2)}\right)^2}d\beta \ dx.
\end{align*}
Now we observe that, using the estimates in Theorem \ref{teo2},
$$
I_{11}\leq -\delta(\|f_0\|_{W^{1,\infty}})\|\partial_x f\|_{H^{1/2}}^2,
$$
where, for uniformly bounded Lipschitz functions,
$$
0<c\leq\delta(\|f_0\|_{W^{1,\infty}}).
$$
Similarly, we have that
\begin{align*}
I_{12}&=\frac{\kappa^+(\rho^--\rho^+)}{4\pi}\int_\TT\text{P.V.}\int_\TT \partial_x f(x)\frac{\partial_x f(x)-\p_{x}f(\beta)}{(\cosh(f(x)-f(\beta))-\cos(x-\beta))^2}\nonumber \\
&\quad  \times \bigg[ \left(\cosh(f(x)-f(\beta))-1\right)\cos(x-\beta) -\sin(x-\beta)\sinh(f(x)-f(\beta))\partial_x f(x)\bigg]\ d \beta \ dx\\
&=-\frac{\kappa^+(\rho^--\rho^+)}{4\pi}\int_\TT\text{P.V.}\int_\TT \partial_x f(\beta)\frac{\partial_x f(x)-\p_{x}f(\beta)}{(\cosh(f(x)-f(\beta))-\cos(x-\beta))^2}\nonumber \\
&\quad  \times \bigg[ \left(\cosh(f(x)-f(\beta))-1\right)\cos(x-\beta) -\sin(x-\beta)\sinh(f(x)-f(\beta))\partial_x f(\beta)\bigg]\ d \beta \ dx\\
&=\frac{\kappa^+(\rho^--\rho^+)}{8\pi}\int_\TT\text{P.V.}\int_\TT \frac{(\partial_x f(x)-\p_{x}f(\beta))^2}{(\cosh(f(x)-f(\beta))-\cos(x-\beta))^2}\nonumber \\
&\quad  \times \bigg[ \left(\cosh(f(x)-f(\beta))-1\right)\cos(x-\beta)\bigg{]}\\
&\quad -\frac{\kappa^+(\rho^--\rho^+)}{8\pi}\int_\TT\text{P.V.}\int_\TT (\partial_x f(x)+\partial_x f(\beta))\frac{(\partial_x f(x)-\p_{x}f(\beta))^2}{(\cosh(f(x)-f(\beta))-\cos(x-\beta))^2}\\
&\quad\times\bigg{[}\sin(x-\beta)\sinh(f(x)-f(\beta))\bigg]\ d \beta \ dx\\
&\leq C(\|f_0\|_{W^{1,\infty}})\|\partial_x f_0\|_{L^\infty}\|\partial_x f\|^2_{H^{1/2}}.
\end{align*}
Changing variables and integrating by parts, we estimate
\begin{align*}
I_2&= -\frac{\kappa^+(\rho^--\rho^+)}{4\pi}\text{P.V.}\int_\TT\int_\TT\frac{(\partial_x f(x))^2}{2} \partial_x\left(\frac{\sin(\beta)}{\cosh(f(x)-f(x-\beta))-\cos(\beta)}\right) \ d\beta \ dx\\
&= \frac{\kappa^+(\rho^--\rho^+)}{4\pi}\text{P.V.}\int_\TT\int_\TT\frac{(\partial_x f(x))^2}{2} \frac{\sin(\beta)\sinh(f(x)-f(x-\beta))(\partial_x f(x)-\partial_x f(x-\beta))}{(\cosh(f(x)-f(x-\beta))-\cos(\beta))^2} \ d\beta \ dx\\
&= \frac{\kappa^+(\rho^--\rho^+)}{4\pi}\text{P.V.}\int_\TT\int_\TT\frac{(\partial_x f(x))^2}{2} \frac{\sin(x-\beta)\sinh(f(x)-f(\beta))(\partial_x f(x)-\partial_x f(\beta))}{(\cosh(f(x)-f(\beta))-\cos(x-\beta))^2} \ d\beta \ dx\\
&= -\frac{\kappa^+(\rho^--\rho^+)}{4\pi}\text{P.V.}\int_\TT\int_\TT\frac{(\partial_x f(\beta))^2}{2} \frac{\sin(x-\beta)\sinh(f(x)-f(\beta))(\partial_x f(x)-\partial_x f(\beta))}{(\cosh(f(x)-f(\beta))-\cos(x-\beta))^2} \ d\beta \ dx\\
&=\frac{\kappa^+(\rho^--\rho^+)}{8\pi}\text{P.V.}\int_\TT\int_\TT\frac{(\partial_x f(x)+\partial_x f(\beta))}{2} \frac{\sin(x-\beta)\sinh(f(x)-f(\beta))(\partial_x f(x)-\partial_x f(\beta))^2}{(\cosh(f(x)-f(\beta))-\cos(x-\beta))^2} \ d\beta \ dx\\
&\leq C(\|f_0\|_{W^{1,\infty}})\|\partial_x f_0\|_{L^\infty}\|\partial_x f\|^2_{H^{1/2}}.
\end{align*}
Integrating by parts, we have that 
\begin{align*}
I_3&= -\frac{\mathcal{A}}{4\pi} \text{P.V.} \int_\TT \int_\TT\int_\TT(\partial_x f(x))^2 \p_{x}\bigg{(}\frac{\sinh(f(x)+h_{2})}{\cosh(f(x)+h_{2})-\cos(w)}  \nonumber \\
& \hspace{2cm} \times  \frac{\sin(z)}{\cosh(h_{2}+f(x-w-z))-\cos(z)}\bigg{)} \  dz \ dw \ dx  +C(\|f_0\|_{W^{1,\infty}})\|\partial_x f\|_{L^2}^2\\
&\leq C(\|f_0\|_{W^{1,\infty}})\|\partial_x f\|_{L^2}^2.
\end{align*}
Similarly,
$$
I_4+I_5\leq C(\|f_0\|_{W^{1,\infty}})\|\partial_x f\|_{L^2}^2.
$$
As a consequence, we find that
$$
\frac{1}{2}\frac{d}{dt}\|\partial_x f\|_{L^2}^2+\delta(\|f_0\|_{W^{1,\infty}})\|\partial_x f\|_{H^{1/2}}^2\leq C(\|f_0\|_{W^{1,\infty}})\|\partial_x f_0\|_{L^\infty}\|\partial_x f\|^2_{H^{1/2}}+C(\|f_0\|_{W^{1,\infty}})\|\partial_x f\|_{L^2}^2.
$$
And as a consequence, due to the smallness of the initial data, we conclude
$$
\frac{d}{dt}\|\partial_x f\|_{L^2}^2+\sigma\|\partial_x f\|_{H^{1/2}}^2\leq C(\|f_0\|_{W^{1,\infty}})\|\partial_x f\|_{L^2}^2,
$$
for a sufficiently small $\sigma$. Using Gronwall tinequality, the previous inequality ensures the following $\varepsilon$-uniform bounds
$$
\|\partial_x f^\varepsilon\|_{L^2}\leq \|\partial_x f_0\|_{L^2}e^{Ct}.
$$
and
$$
\int_0^T\|\partial_x f^\varepsilon(s)\|^2_{H^{1/2}}ds\leq \frac{\|\partial_x f_0\|_{L^2}}{\sigma}e^{Ct}.
$$

\textbf{Step 4 (Passing to the limit):} We have the following uniform bounds
$$
\|f^\varepsilon\|_{L^\infty}\leq \|f_0\|_{L^\infty},
$$
$$
\|\partial_x f^\varepsilon\|_{L^\infty}\leq \|\partial_x f_0\|_{L^\infty},
$$
$$
\|\partial_x f^\varepsilon\|_{L^2}\leq \|\partial_x f_0\|_{L^2}e^{Ct},
$$
$$
\int_0^T\|\partial_x f^\varepsilon(s)\|^2_{H^{1/2}}ds\leq \frac{\|\partial_x f_0\|_{L^2}}{\sigma}e^{Ct}.
$$
Due to Banach-Alaoglu Theorem, we obtain that there exists a subsequence (denoted again by $f^\varepsilon$) that satisfies the following 
$$
f^\varepsilon \stackrel{\ast}{\rightharpoonup} f \in L^\infty(0,T;L^\infty),
$$
$$
\partial_xf^\varepsilon \stackrel{\ast}{\rightharpoonup} \partial_x f \in L^\infty(0,T;L^\infty),
$$
$$
\partial_xf^\varepsilon \rightharpoonup \partial_x f \in L^2(0,T;H^{1/2}),
$$
where 
$$
f\in L^\infty(0,T;W^{1,\infty})\cap L^2(0,T;H^{3/2}). 
$$
Using duality, we find that 
$$
\|\partial_tf^\varepsilon\|_{H^{-1}}=\sup_{\psi\in H^1}\langle \partial_t f^\varepsilon,\psi \rangle\leq C(\|f_0\|_{W^{1,\infty}}).
$$
Due to a classical Aubin-Lions type argument (see Corollary 4 of \cite{simon1986compact}), we also conclude the following strong convergence
$$
\partial_xf^\varepsilon \rightarrow \partial_x f \in L^2(0,T;H^{1/2-\epsilon}), \forall \epsilon>0.
$$
Equipped with these convergences we can conclude that the limit function $f$ is indeed a weak solution to the inhomogeneous Muskat problem \eqref{eq:I} in the sense of definition \eqref{eq:weak} (see \cite{gancedo2020global2} for more details).
\end{proof}

\section*{Acknowledgments}
D.A-O is supported by the Alexander von Humboldt Foundation. R.G-B is supported by the project ``Mathematical Analysis of Fluids and Applications" with reference PID2019-109348GA-I00/AEI/ 10.13039/501100011033 and acronym ``MAFyA" funded by Agencia Estatal de Investigaci\'on and the Ministerio de Ciencia, Innovacion y Universidades (MICIU). Project supported by a 2021 Leonardo Grant for Researchers and Cultural Creators, BBVA Fundation. The BBVA Foundation accepts no responsability for the opinions, statements and contents included in the project and/or the results thereof, which are entirely the responsability of the authors.
%%%%%%%%%%%%%%%%%%%%%%%%%%%% 
%%%%%%%%% BIBLIOGRAPHY %%%%%%%%%%
%%%%%%%%%%%%%%%%%%%%%%%%%%%%%

\bibliographystyle{abbrv}
%\bibliography{references}

\begin{thebibliography}{10}

\bibitem{alazard2020paralinearization}
T.~Alazard and O.~Lazar.
\newblock Paralinearization of the muskat equation and application to the
  cauchy problem.
\newblock {\em Archive for Rational Mechanics and Analysis}, 237(2):545--583,
  2020.

\bibitem{alazard2020endpoint}
T.~Alazard and Q.-H. Nguyen.
\newblock Endpoint Sobolev theory for the Muskat equation.
\newblock {\em arXiv preprint arXiv:2010.06915}, 2020.

\bibitem{alazard2021cauchy2}
T.~Alazard and Q.-H. Nguyen.
\newblock On the Cauchy problem for the Muskat equation. ii: Critical initial
  data.
\newblock {\em Annals of PDE}, 7(1):1--25, 2021.

\bibitem{alazard2021cauchy}
T.~Alazard and Q.-H. Nguyen.
\newblock On the Cauchy problem for the Muskat equation with non-Lipschitz
  initial data.
\newblock {\em Communications in Partial Differential Equations}, pages 1--42,
  2021.

\bibitem{alazard2021quasilinearization}
T.~Alazard and Q.-H. Nguyen.
\newblock Quasilinearization of the 3d Muskat equation, and applications to the
  critical Cauchy problem.
\newblock {\em arXiv preprint arXiv:2103.02474}, 2021.

\bibitem{ambrose2004well}
D.~M. Ambrose.
\newblock Well-posedness of two-phase Hele--Shaw flow without surface tension.
\newblock {\em European Journal of Applied Mathematics}, 15(5):597--607, 2004.

\bibitem{ambrose2014zero}
D.~M. Ambrose.
\newblock The zero surface tension limit of two-dimensional interfacial Darcy
  flow.
\newblock {\em Journal of Mathematical Fluid Mechanics}, 16(1):105--143, 2014.

\bibitem{berselli2014local}
L.~C. Berselli, D.~C{\'o}rdoba, and R.~Granero-Belinch{\'o}n.
\newblock Local solvability and turning for the inhomogeneous Muskat problem.
\newblock {\em Interfaces and Free Boundaries}, 16(2):175--213, 2014.

\bibitem{cameron2018global}
S.~Cameron.
\newblock Global well-posedness for the two-dimensional Muskat problem with
  slope less than 1.
\newblock {\em Analysis \& PDE}, 12(4):997--1022, 2018.

\bibitem{cameron2020eventual}
S.~Cameron.
\newblock Eventual regularization for the 3d Muskat problem: Lipschitz for
  finite time implies global existence.
\newblock {\em arXiv preprint arXiv:2007.03099}, 2020.

\bibitem{cameron2020global}
S.~Cameron.
\newblock Global wellposedness for the 3d Muskat problem with medium size
  slope.
\newblock {\em arXiv preprint arXiv:2002.00508}, 2020.

\bibitem{castro2013breakdown}
{\'A}.~Castro, D.~C{\'o}rdoba, C.~Fefferman, and F.~Gancedo.
\newblock Breakdown of smoothness for the Muskat problem.
\newblock {\em Archive for Rational Mechanics and Analysis}, 208(3):805--909,
  2013.

\bibitem{castro2012rayleigh}
{\'A}.~Castro, D.~C{\'o}rdoba, C.~Fefferman, F.~Gancedo, and
  M.~L{\'o}pez-Fern{\'a}ndez.
\newblock Rayleigh-taylor breakdown for the Muskat problem with applications to
  water waves.
\newblock {\em Annals of mathematics}, pages 909--948, 2012.

\bibitem{castro2016splash}
{\'A}.~Castro~Mart{\'\i}nez, D.~C{\'o}rdoba~Gazolaz, C.~L. Fefferman, and
  F.~Gancedo~Garc{\'\i}a.
\newblock Splash singularities for the one-phase Muskat problem in stable
  regimes.
\newblock {\em Archive for Rational Mechanics and Analysis, 222 (1), 213-243.},
  2016.

\bibitem{chen2021muskat}
K.~Chen, Q.-H. Nguyen, and Y.~Xu.
\newblock The Muskat problem with $C^1$ data.
\newblock {\em arXiv preprint arXiv:2103.09732}, 2021.

\bibitem{cheng2016well}
C.~A. Cheng, R.~Granero-Belinch{\'o}n, and S.~Shkoller.
\newblock Well-posedness of the Muskat problem with $H^2$ initial data.
\newblock {\em Advances in Mathematics}, 286:32--104, 2016.

\bibitem{constantin2016muskat}
P.~Constantin, D.~C{\'o}rdoba, F.~Gancedo, L.~Rodr{\'\i}guez-Piazza, and R.~M.
  Strain.
\newblock On the Muskat problem: global in time results in 2d and 3d.
\newblock {\em American Journal of Mathematics}, 138(6):1455--1494, 2016.

\bibitem{constantin2012global}
P.~Constantin, D.~C{\'o}rdoba, F.~Gancedo, and R.~M. Strain.
\newblock On the global existence for the Muskat problem.
\newblock {\em Journal of the European Mathematical Society}, 15(1):201--227,
  2012.

\bibitem{constantin2017global}
P.~Constantin, F.~Gancedo, R.~Shvydkoy, and V.~Vicol.
\newblock Global regularity for 2d Muskat equations with finite slope.
\newblock {\em Ann. Inst. H. Poincare Anal. Non Lineaire} 34(4):1041--1074, 2017.

\bibitem{cordoba2011interface}
A.~C{\'o}rdoba, D.~C{\'o}rdoba, and F.~Gancedo.
\newblock Interface evolution: the Hele-Shaw and Muskat problems.
\newblock {\em Annals of mathematics}, pages 477--542, 2011.

\bibitem{cordoba2013porous}
A.~C{\'o}rdoba, D.~C{\'o}rdoba, and F.~Gancedo.
\newblock Porous media: the Muskat problem in three dimensions.
\newblock {\em Analysis \& PDE}, 6(2):447--497, 2013.

\bibitem{cordoba2007contour}
D.~C{\'o}rdoba and F.~Gancedo.
\newblock Contour dynamics of incompressible 3-d fluids in a porous medium with
  different densities.
\newblock {\em Communications in Mathematical Physics}, 273(2):445--471, 2007.

\bibitem{cordoba2009maximum}
D.~C{\'o}rdoba and F.~Gancedo.
\newblock A maximum principle for the Muskat problem for fluids with different
  densities.
\newblock {\em Communications in Mathematical Physics}, 286(2):681--696, 2009.

\bibitem{cordoba2015note}
D.~C{\'o}rdoba, J.~G{\'o}mez-Serrano, and A.~Zlato{\v{s}}.
\newblock A note on stability shifting for the Muskat problem.
\newblock {\em Philosophical Transactions of the Royal Society A: Mathematical,
  Physical and Engineering Sciences}, 373(2050):20140278, 2015.

\bibitem{cordoba2017note}
D.~C{\'o}rdoba, J.~G{\'o}mez-Serrano, and A.~Zlato{\v{s}}.
\newblock A note on stability shifting for the Muskat problem, ii: From stable
  to unstable and back to stable.
\newblock {\em Analysis \& PDE}, 10(2):367--378, 2017.

\bibitem{cordoba2018global}
D.~Cordoba and O.~Lazar.
\newblock Global well-posedness for the 2d stable Muskat problem in $ h^{3/2}
  $.
\newblock {\em arXiv preprint arXiv:1803.07528}, 2018.

\bibitem{cordoba2014confined}
D.~C{\'o}rdoba~Gazolaz, R.~Granero-Belinch{\'o}n, and R.~Orive-Illera.
\newblock The confined Muskat problem: Differences with the deep water regime.
\newblock {\em Communications in Mathematical Sciences}, 12(3):423--455, 2014.

\bibitem{escher2011generalized}
J.~Escher, A.-V. Matioc, and B.-V. Matioc.
\newblock A generalized Rayleigh--Taylor condition for the Muskat problem.
\newblock {\em Nonlinearity}, 25(1):73, 2011.

\bibitem{escher2011parabolicity}
J.~Escher and B.-V. Matioc.
\newblock On the parabolicity of the Muskat problem: Well-posedness, fingering,
  and stability results.
\newblock {\em Zeitschrift f{\"u}r Analysis und ihre Anwendungen},
  30(2):193--218, 2011.

\bibitem{escher2015domain}
Joachim Escher, Bogdan-Vasile Matioc, and Christoph Walker.
\newblock The domain of parabolicity for the {M}uskat problem.
\newblock {\em Indiana Univ. Math. J.}, 67(2):679--737, 2018.

\bibitem{gancedo2017survey}
F.~Gancedo.
\newblock A survey for the Muskat problem and a new estimate.
\newblock {\em SeMA Journal}, 74(1):21--35, 2017.

\bibitem{gancedo2019muskat}
F.~Gancedo, E.~Garcia-Juarez, N.~Patel, and R.~M. Strain.
\newblock On the Muskat problem with viscosity jump: global in time results.
\newblock {\em Advances in Mathematics}, 345:552--597, 2019.

\bibitem{gancedo2020global2}
F.~Gancedo, R.~Granero-Belinch{\'o}n, and S.~Scrobogna.
\newblock Global existence in the Lipschitz class for the N-Peskin problem.
\newblock {\em To appear in Indiana University Mathematics Journal,  arXiv preprint arXiv:2011.02294}, 2020.

\bibitem{gancedo2020surface}
F.~Gancedo, R.~Granero-Belinch{\'o}n, and S.~Scrobogna.
\newblock Surface tension stabilization of the Rayleigh-Taylor instability for
  a fluid layer in a porous medium.
\newblock 37(6):1299--1343, 2020.

\bibitem{gancedo2020global}
F.~Gancedo and O.~Lazar.
\newblock Global well-posedness for the 3d Muskat problem in the critical
  Sobolev space.
\newblock {\em arXiv preprint arXiv:2006.01787}, 2020.

\bibitem{gomez2014turning}
J.~G{\'o}mez-Serrano and R.~Granero-Belinch{\'o}n.
\newblock On turning waves for the inhomogeneous Muskat problem: a
  computer-assisted proof.
\newblock {\em Nonlinearity}, 27(6):1471, 2014.

\bibitem{granero2014global}
R.~Granero-Belinch{\'o}n.
\newblock Global existence for the confined Muskat problem.
\newblock {\em SIAM Journal on Mathematical Analysis}, 46(2):1651--1680, 2014.

\bibitem{granero2020growth}
R.~Granero-Belinch{\'o}n and O.~Lazar.
\newblock Growth in the Muskat problem.
\newblock {\em Mathematical Modelling of Natural Phenomena}, 15:7, 2020.

\bibitem{granero2019well}
R.~Granero-Belinch{\'o}n and S.~Shkoller.
\newblock Well-posedness and decay to equilibrium for the Muskat problem with
  discontinuous permeability.
\newblock {\em Transactions of the American Mathematical Society},
  372(4):2255--2286, 2019.

\bibitem{matioc2019well}
A.-V. Matioc and B.-V. Matioc.
\newblock Well-posedness and stability results for a quasilinear periodic
  Muskat problem.
\newblock {\em Journal of Differential Equations}, 266(9):5500--5531, 2019.

\bibitem{matioc2018viscous}
B.-V. Matioc.
\newblock Viscous displacement in porous media: the Muskat problem in 2d.
\newblock {\em Transactions of the American Mathematical Society},
  370(10):7511--7556, 2018.

\bibitem{nguyen2020paradifferential}
H.~Q. Nguyen and B.~Pausader.
\newblock A paradifferential approach for well-posedness of the Muskat problem.
\newblock {\em Archive for Rational Mechanics and Analysis}, 237(1):35--100,
  2020.

\bibitem{patel2021global}
N.~Patel and N.~Shankar.
\newblock Global results for the inhomogeneous Muskat problem.
\newblock {\em arXiv preprint arXiv:2102.07754}, 2021.

\bibitem{patel2017large}
N.~Patel and R.~M. Strain.
\newblock Large time decay estimates for the Muskat equation.
\newblock {\em Communications in Partial Differential Equations},
  42(6):977--999, 2017.

\bibitem{pernas2017local}
T.~Pernas-Casta{\~n}o.
\newblock Local-existence for the inhomogeneous Muskat problem.
\newblock {\em Nonlinearity}, 30(5):2063, 2017.

\bibitem{pruess2016muskat}
Jan Pruess and Gieri Simonett.
\newblock On the Muskat flow.
\newblock {\em Evolution Equations and Control Theory}, 5:631--645, 2016.

\bibitem{siegel2004global}
M.~Siegel, R.~E. Caflisch, and S.~Howison.
\newblock Global existence, singular solutions, and ill-posedness for the
  Muskat problem.
\newblock {\em Communications on Pure and Applied Mathematics}, 57(10):1374--1411, 2004.

\bibitem{simon1986compact}
J.~Simon.
\newblock Compact sets in the space $L^p (0, t; B)$.
\newblock {\em Annali di Matematica pura ed applicata}, 146(1):65--96, 1986.

\end{thebibliography}

\end{document}